\documentclass{amsart}

\title{Limits of multivariate elliptic beta integrals and related bilinear forms}
\author{Fokko J. van de Bult and Eric M. Rains}

\usepackage{fullpage}
\usepackage{tikz} 
\usepackage{multirow}

\newtheorem{theorem}{Theorem}[section]
\newtheorem{definition}[theorem]{Definition}
\newtheorem{proposition}[theorem]{Proposition} 
\newtheorem{corollary}[theorem]{Corollary}
\newtheorem{lemma}[theorem]{Lemma}

\newcommand{\II}{I\!I}
\DeclareMathOperator{\val}{val}
\DeclareMathOperator{\lc}{lc}

\begin{document}

\begin{abstract}
In this article we consider the elliptic Selberg integral, which is a $BC_n$ symmetric multivariate extension of the elliptic beta integral. We categorize the limits that are obtained as $p\to 0$, for given behavior of the parameters as $p\to 0$. This article is therefore the multivariate version of \cite{vdBR}. The integrand of the elliptic Selberg integral is the measure for the biorthogonal functions from \cite{Rainstrafo}, so we also consider the limits of the associated bilinear form. We also provide the limits for the discrete version of this bilinear form, which is related to a multivariate extension of the Frenkel-Turaev summation.
\end{abstract}

\maketitle

Elliptic hypergeometric functions have been a popular area of study since the publication \cite{FT} of Frenkel and Turaev's summation formula. Elliptic hypergeometric series are a generalization of hypergeometric series, where the quotient of two subsequent terms is an elliptic function of $n$ instead of a rational function in $n$, respectively a rational function in $q^n$, for ordinary, respectively basic, hypergeometric functions. Just as in the case of other classes of hypergeometric series, there also exist closely related elliptic hypergeometric integrals, which involve the elliptic gamma function \cite{egf}. The most important of these is a generalization of the beta integral \cite{Spirbeta}.

Many identities for ordinary and basic hypergeometric functions can be generalized to the elliptic hypergeometric setting. 
One of these results is the construction of a family of biorthogonal functions, which are biorthogonal both with respect to the Frenkel-Turaev summation formula \cite{SZ1}, \cite{SZ2}, and with respect to the elliptic beta integral \cite{S}. These biorthogonal functions are generalizations of the (Askey-)Wilson polynomials. These biorthogonal functions have been generalized to the multivariate ($BC_n$-symmetric) setting, discrete measure in \cite{RainsBCn}, continuous measure in \cite{Rainstrafo}. The multivariate biorthogonal functions are an elliptic analogue of both the Koornwinder polynomials and the Macdonald polynomials. The associated generalization of the measures give multivariate analogues of the Frenkel-Turaev summation formula and the elliptic beta integral, the latter is called the elliptic Selberg integral. 

It is well known that upon taking a proper limit in elliptic hypergeometric functions (i.e., letting the parameter $p$ go to 0), one obtains basic hypergeometric functions. In this way we can make precise the statement that certain elliptic hypergeometric identities are generalizations of corresponding basic hypergeometric identities. The limit obtained depends, as one would expect, on how all the other parameters behave as $p\to 0$. Studying these different possible limits reveals the structure behind the multitudinous different basic hypergeometric identities. This project has been carried out by the authors for the univariate elliptic beta integral \cite{vdBR}, univariate biorthogonal functions \cite{vdBRuniv}, and multivariate biorthogonal functions \cite{vdBRmult}. 

In this last paper we only considered the limits of the biorthogonal functions themselves, not of the related measures. In particular we now have families of functions which are formally biorthogonal, but without a proper measure with respect to which they are biorthogonal. In this article we consider the limits of the elliptic hypergeometric measures and find at least one measure for each of the limiting families of \cite{vdBRmult}. In the process we obtain the limits of the multivariate Frenkel-Turaev summation and the elliptic Selberg integral (the multivariate beta integral), which are the normalization constants in the measure. 

As in \cite{vdBRmult} we only consider limits in which the parameter $t$, controlling the cross-terms of the multivariate Selberg integral, remains constant as $p\to 0$. One might well expect that there are still other interesting limits to be found in which $t$ does depend non-trivially on $p$. Fixing $t$ ensures that the combinatorics behind the different limiting measures/integrals we obtain here is identical to the combinatorics of the univariate limits given in \cite{vdBR}, just as the combinatorics for the limits of the multivariate biorthogonal functions from \cite{vdBRmult} equals the combinatorics of the univariate limits in \cite{vdBRuniv}.

The article is organized as follows. In Section \ref{secnot} we will introduce the notation, and present basic properties of the basic functions. In Section \ref{secdisc} we consider the limits of the discrete measures, that is the multivariate Frenkel-Turaev summation and its associated bilinear form. The following section discusses this same situation but with more parameters (i.e., higher ``$m$'' in the notation of \cite{Rainstrafo}). Section \ref{secac} then considers the continuous measures, that is the bilinear forms associated to the elliptic Selberg integral, with the next section giving some limits to bilateral series of this continuous measure. In Section \ref{sec6} we consider the limits of the associated integral with more variables. Finally Appendix \ref{apb} contains some tedious calculations which are necessary to obtain the series limits in Section \ref{secac} and the bilateral limits.

\section{Notation}\label{secnot}
The notations we use are identical to the notations from \cite{vdBRmult}. 

\subsection{Univariate $q$-symbols}

We say a function $f(x;z)$ is written multiplicatively in $x$ if the presence of multiple parameters at the place of $x$ indicates a product; and if $\pm$ symbols in those parameters also indicate a product over all possible combinations of $+$ and $-$ signs. For example
\begin{align*}
f(x_1, x_2, \ldots, x_n;z) &= \prod_{r=1}^n f(x_i;z), \\ f(x^{\pm 1} y^{\pm 1};z) &= 
f(xy;z)f(x/y;z)f(y/x;z)f(1/xy;z).
\end{align*}

Now we define the $q$-symbols and their elliptic analogues as in \cite{GR}. Let $0<|q|,|p|<1$ and set
\begin{align*}
(x;q) &= \prod_{r=0}^\infty (1-xq^r), &
(x;q)_m &= \prod_{r=0}^{m-1} (1-xq^r), & (x;p,q) &= \prod_{r,s\geq 0} (1-xp^rq^s) \\
\theta(x;p) &= (x,p/x;p),&
\theta(x;q;p)_m &= \prod_{r=0}^{m-1} \theta(xq^r;p), & 
\Gamma(x;p,q) & = \prod_{i,j\geq 0} \frac{1-p^{i+1}q^{j+1}/x}{1-p^i q^j x}.
\end{align*}
All these functions are written multiplicatively in $x$. Note that the terminating product $(x;q)_m$ is also defined if $|q|\geq 1$. Likewise $\theta(x;q;p)_m$ is defined for all $q$, though we must still insist on $|p|<1$.

%

%
%

\subsection{Partitions}
We use the notations of \cite{RainsBCn} for partitions, which is the notation from Macdonald's book \cite{Macdonald} with some additions. 
If $\lambda \subset m^n$ then we write $m^n-\lambda$ for the complementary partition, given by 
\[
(m^n-\lambda)_i = \begin{cases} m-\lambda_{n+1-i}  & 1\leq i\leq n \\ 0 & i>n \end{cases}
\]
%

Some convenient numbers associated with $\lambda$ are
\begin{align*}
|\lambda|&=\sum_i \lambda_i \\
n(\lambda) &= \sum_{i} \lambda_i (i-1)= \sum_{(i,j)\in \lambda} (i-1) = \sum_{j} \binom{\lambda_j'}{2}  = \frac12 \sum_{(i,j) \in \lambda} (\lambda_j'-1).
\end{align*}
Here we use $\sum_{(i,j)\in \lambda}$, which means we sum over all boxes in the Young diagram, i.e. we sum over 
$1\leq i\leq l(\lambda)$ and $1\leq j\leq \lambda_i$. A similar notation is used for products. Notice that we can extend these definitions for arbitrary $\lambda \in \mathbb{Z}^n$, by using the first definition for $n(\lambda)$. We can also define $n(\lambda')$ for arbitrary $\lambda \in \mathbb{Z}^n$ by using the third equation, i.e. 
$n(\lambda') =  \sum_{i} \binom{\lambda_i}{2}$, in particular we do not need to define $\lambda'$ itself to define $n(\lambda')$.

In the entire article we will use $n$ for the number of variables $z_i$, which means that our partitions usually satisfy $\ell(\lambda)\leq n$. (From context it should always be clear when we use $n$ as number of variables and when we use it for the function $n(\lambda)$.)

\subsection{Multivariate $q$-symbols}

Let us now define the $C$-symbols (also written multiplicatively in $x$).
\begin{align}
\label{eqdefc0}C_{\lambda}^0(x;q,t;p) &= \prod_{(i,j)\in\lambda} \theta(q^{j-1}t^{1-i}x;p) &\tilde C_{\lambda}^0(x;q,t) &= \prod_{(i,j)\in\lambda} (1-q^{j-1}t^{1-i}x) \\
\label{eqdefcm}C_{\lambda}^-(x;q,t;p) &= \prod_{(i,j)\in\lambda} \theta(q^{\lambda_i-j}t^{\lambda_j'-i}x;p) & \tilde C_{\lambda}^-(x;q,t) &= \prod_{(i,j)\in\lambda} (1-q^{\lambda_i-j}t^{\lambda_j'-i}x)\\
\label{eqdefcp} C_{\lambda}^+(x;q,t;p) &= \prod_{(i,j)\in\lambda} \theta(q^{\lambda_i+j-1}t^{2-\lambda_j'-i}x;p) &
\tilde C_{\lambda}^+(x;q,t) &= \prod_{(i,j)\in\lambda} (1-q^{\lambda_i+j-1}t^{2-\lambda_j'-i}x)
\end{align}
The elliptic $C_{\lambda}$ are as in \cite{RainsBCn}, while the $\tilde C_{\lambda}$ are the 
$C_{\lambda}$ from \cite{RainsBCnpoly}. Observe that the following alternative expressions
\begin{align}\label{eqaltdefcl}
C_{\lambda}^0(x;q,t;p) & = \prod_{i=1}^n \theta(t^{1-i}x;p)_{\lambda_i}, \qquad \qquad 
C_{\lambda}^-(x;q,t;p) = \prod_{1\leq i<j\leq n} \frac{\theta(t^{j-1-i}x;p)_{\lambda_i-\lambda_j}}{\theta(t^{j-i}x;p)_{\lambda_i-\lambda_j}}
\prod_{i=1}^n  \theta(t^{n-i}x;p)_{\lambda_i}, \\
C_{\lambda}^+(x;q,t;p) &= \prod_{1\leq i<j\leq n} \frac{\theta(t^{2-j-i}x;p)_{\lambda_i+\lambda_j}}{\theta(t^{3-j-i}x;p)_{\lambda_i+\lambda_j}}
\prod_{i=1}^n \frac{\theta(t^{2-2i}x;p)_{2\lambda_i}}{\theta(t^{2-n-i}x;p)_{\lambda_i}}. \nonumber 
\end{align}
hold (they are equivalent to \cite[(1.11)-(1.13)]{Rainstrafo}), 
 which allows us to define these functions for arbitrary $\lambda \in \mathbb{Z}^n$, given the usual convention 
 $\theta(x;p)_{-n} = 1/\prod_{i=1}^n \theta(xq^{-i};p)$ for $n\geq 0$. Similar expressions with $\theta(x)$ replaced by $(1-x)$ give definitions for the $\tilde C_{\lambda}^\epsilon$ for $\lambda \in \mathbb{Z}^n$.

%
%
%
%
%
The $\Delta$-symbols are defined by
\[
\Delta_\lambda^0(a~|~b;q,t;p) = \frac{C_{\lambda}^0(b;q,t;p)}{C_\lambda^0(pqa/b;q,t;p)},
\] 
which is written multiplicatively in $b$ and
\[
\Delta_\lambda(a~|~b_1,\ldots, b_r;q,t;p) = \Delta_\lambda^0(a~|~b_1, \ldots, b_r;q,t;p) \frac{C_{2\lambda^2}^0(pqa;q,t;p)}{C_\lambda^-(pq,t;q,t;p) C_{\lambda}^+(a,pqa/t;q,t;p)}
\]
which is emphatically not written multiplicatively. Here $2\lambda^2$ denotes the partition with $(2\lambda^2)_i = 2(\lambda_{\lceil i/2\rceil})$. A $q$-hypergeometric version of $\Delta_{\lambda}$ is defined by 
\[
\tilde \Delta_{\lambda}^{(n)}(a;q,t) = \frac{\tilde C_{2\lambda^2}^0(aq;q,t) \tilde C_{\lambda}^0(t^n;q,t)}{
\tilde C_{\lambda}^0(aq/t^n) \tilde C_{\lambda}^-(q,t;q,t) \tilde C_{\lambda}^+(a,aq/t;q,t)}
\left(-\frac{1}{a^2q^2t^{n-1}}\right)^{|\lambda|} q^{-3n(\lambda')} t^{5 n(\lambda)}.
\]
Whenever no confusion is possible we omit the $;q,t;p$ or the $;q,t$ from the arguments.

The $C_\lambda$'s are multivariate analogues of the theta Pochhammer symbols, while the $\tilde C_{\lambda}$'s are multivariate versions of $q$-Pochhammer symbols. The $\Delta_{\lambda}$ and $\tilde \Delta_{\lambda}$ correspond univariately to the summands of a very well poised series, indeed 
%
\begin{align*}
\Delta_l(a~|~b_1,\ldots,b_r;q,t;p) &=
\frac{\theta(apq^{2l};p)}{\theta(ap;p)}
\frac{\theta(ap,\frac{ap^2q}{t};q;p)_l}{\theta(q,\frac{t}{p};q;p)_l}
\prod_{s=1}^r \frac{\theta(b_s;q;p)_l}{\theta(\frac{pqa}{b_s};q;p)_l} q^l, \\ 
\tilde \Delta_{l}^{(1)}(a;q,t) &= \frac{1-aq^{2l}}{1-a} \frac{(a;q)_{l}}{(q;q)_l} \left( -\frac{1}{a^2q^2}\right)^l q^{-3\binom{l}{2}}.
\end{align*}
\subsection{Transformations of generalized $q$-symbols}
It is convenient to write down a few elementary transformation formulas for these functions, analogues of some identities for theta Pochhammer symbols. 
The following expressions can all be obtained from the two elementary symmetries $\theta(px;p)=
-\frac{1}{x}\theta(x;p)$.
\begin{align} 
C_{\lambda}^0(px) & = C_{\lambda}^0(x) \left(-\frac{1}{x}\right)^{|\lambda|} q^{-n(\lambda')} t^{n(\lambda)}, \label{eqc0p}  \\
C_{\lambda}^-(px) & = C_{\lambda}^-(x) \left(-\frac{1}{x}\right)^{|\lambda|} q^{-n(\lambda')} t^{-n(\lambda)}, \label{eqcmp} \\
C_{\lambda}^+(px) & =  C_{\lambda}^+(x) \left(-\frac{1}{qx}\right)^{|\lambda|} q^{-3n(\lambda')} t^{3n(\lambda)}. \label{eqcpp}
\end{align}

Likewise we can find shifting formulas for the $\Delta$ functions:
\begin{align}
\Delta_{\lambda}^0(a~|~pb, \ldots, v_i,\ldots) & = \Delta_{\lambda}^0(a~|~b,\ldots,v_i,\ldots) \left( \frac{1}{aq}\right)^{|\lambda|} q^{-2n(\lambda')}t^{2n(\lambda)} \\
\Delta_{\lambda}^0(\frac{a}{p}~|~ b_1, \ldots, b_r) &= \Delta_{\lambda}^0(a~|~b_1,\ldots, b_r)\left(\frac{\prod_i b_i}{(-aq)^r}\right)^{|\lambda|}
q^{-r n(\lambda')} t^{r n(\lambda)} \\
\Delta_{\lambda}(a~|~pb, \ldots, v_i, \ldots) &= \Delta_{\lambda}(a~|~b,\ldots,v_i,\ldots) \left( \frac{1}{aq}\right)^{|\lambda|} 
q^{-2n(\lambda')}t^{2n(\lambda)} \\
\Delta_{\lambda}(\frac{a}{p}~|~ b_1, \ldots, b_r) &= \Delta_{\lambda}(a~|~b_1,\ldots, b_r)\left(\frac{pq}{t} \frac{\prod_i b_i}{(-aq)^{r-2}}\right)^{|\lambda|}
q^{(2-r) n(\lambda')} t^{(r-2) n(\lambda)}.
\end{align}
We would like to remark that $\Delta_{\lambda}^0(a~|~b_1,\ldots,b_{r})$ is
invariant if we multiply each individual $b_j$ by an integer multiple of
$p$, while keeping the product $\prod_r b_r$ fixed. Moreover, if $r$ is
even, then $\Delta_{\lambda}^0$ is invariant if we multiply $a$ and the
$b_j$'s by integer multiples of $p$, as long as the balancing condition
$\prod_i b_i = (apq)^{r/2}$ holds (both before and after the
$p$-shift). Similarly, as long as the balancing condition $pq \prod_{i}
b_i=t (apq)^k$ holds $\Delta_{\lambda}(a~|~b_1,\ldots,b_{2k+2})$ remains
invariant under multiplication of the parameters by integer powers of $p$.

\subsection{Power series in $p$}
Most functions we are interested in are elements of the field $M(x)$, defined in \cite[Section 2]{vdBRuniv}. This is a field of (multivariate) meromorphic functions in the variables $x=(x_1,x_2,\ldots)$, which can be expressed as power series 
$f=\sum_{t\in T} a_t(x) p^t$ for some discrete set $T$, which is bounded from below, with coefficients $a_t$, which are rational functions in $x$. The valuation of such a series is $\val(f) = \min_{t\in T}t$ and the leading coefficient is given by $\lc(f) = a_{\val(t)}$. Since we are interested in the behavior as $p\to 0$, we think of the valuation as describing the size of $f$ as $p\to 0$, while the leading coefficient gives the limit of $f$ (after proper rescaling). The conditions on the space imply that this limit is always uniform on compact sets outside the zero-set of some polynomial in $x$. Moreover, due to some extra conditions placed on the rational functions $a_t$ we obtained the following iterated limit property \cite[Proposition 2.3]{vdBRuniv}.
\begin{proposition}\label{propiteratedlim}
Let $f\in M(x)$, write $p^u x = (p^{u_1}x_1,p^{u_2}x_2,\ldots)$. Then for small enough $\epsilon>0$ and any $u$ we have
\[
\lc(\lc(f)(p^u x)) = \lc(f(p^{\epsilon u}x)), \qquad \val(f)+ \epsilon \  \val(\lc(f)(p^u x)) = \val(f(p^{\epsilon u} x)).
\]
\end{proposition}
As a corollary we obtain the following important result on the valuation of a sum of two terms
\begin{corollary}\label{corsumequalval}
Let $f,g\in M(x)$ and define $h=f+g$. 
\begin{itemize}
\item If $\val(f)<\val(g)$, then $\val(h)=\val(f)$ and $\lc(h)=\lc(f)$. 
\item If $\val(f)=\val(g)$, and there exists a $u$ such that 
for all small enough $\epsilon>0$ we have $\val(f(p^{\epsilon u}x))< \val(g(p^{\epsilon u}x))$. Then $\val(h)=\val(f)$ and $\lc(h)=\lc(f)+\lc(g)$. 
\end{itemize}
\end{corollary}

\subsection{Limits of generalized $q$-symbols} 
Of course the $q$-symbols discussed before are elements of the field $M(x)$, and since every function appearing in this article is build using these $q$-symbols, they are elements of $M(x)$ as well. Let us now discuss the valuations and leading coefficients of the elliptic $q$-symbols.
%

For ordinary theta functions we have
\[
\val(\theta(xp^{\alpha};p)) = \frac12 \{\alpha\}(\{\alpha\}-1)-\frac12 \alpha(\alpha-1), \qquad 
\lc(\theta(xp^{\alpha};p)) = \begin{cases} (1-x) \left(-\frac{1}{x}\right)^{ \alpha } & \alpha \in \mathbb{Z} \\
\left(-\frac{1}{x}\right)^{\lfloor \alpha \rfloor} & \alpha \not \in \mathbb{Z}, \end{cases}
\]
where $\{\alpha\} = \alpha-\lfloor \alpha \rfloor$ denotes the fractional part of $\alpha$. Note that $\val(\theta(xp^{\alpha};p))$ is a continuous piecewise linear function in $\alpha$. The valuations and leading coefficients of the $C$-symbols are direct consequences of this. While a general formula is easily given, it is rather complex. Thus we refer to the shifting formulas
\eqref{eqc0p} to note that it suffices to give the results for $0\leq \alpha<1$. In that case we get 
\begin{equation}\label{eqlimc}
\val(C_{\lambda}^{\epsilon}(xp^{\alpha})=|\lambda| (\frac12\{\alpha\}(\{\alpha\}-1) - \frac12\alpha(\alpha-1)), \quad (\alpha\in \mathbb{R}), \qquad 
\lc(C_{\lambda}^{\epsilon}(xp^{\alpha})=\begin{cases} \tilde C_{\lambda}^{\epsilon}(x) & \alpha=0 \\\
1 & 0<\alpha<1, \end{cases}
\end{equation}
where $\epsilon=0$, $+$, or $-$.

To take limits of $\Delta_{\lambda}^0$ it is often most convenient to express it in terms of $C_{\lambda}^0$, and take the limits of the $C_{\lambda}^0$'s. One of the important reasons we so often use the $\Delta_{\lambda}^0$ is that it is elliptic (under the balancing condition given above). After taking the limit, we cannot shift by $p$ anymore, so ellipticity becomes a non-existent concept, thus removing the usefulness of this notation.

As for $\Delta_{\lambda}$ we'll only consider $\Delta_{\lambda}(a p^{\alpha}~|~t^n ;q,t;p)$. It turns out that every instance of $\Delta_{\lambda}$ we encounter has $t^n$ as one of its $b$-parameters. Moreover the quotient of any $\Delta_{\lambda}$ and this one is a $\Delta_{\lambda}^0$ and we can express its limits in terms of $\tilde C_{\lambda}^0$'s as described above. Thus writing down the valuation and leading coefficient of this specific $\Delta_{\lambda}$ suffices to be able to obtain the limits of the general case. We assume $\ell(\lambda) \leq n$, as otherwise $\Delta_{\lambda}(a ~|~ t^n)=0$ identically.
%
\begin{align*}
\val(\Delta_{\lambda}(ap^{\alpha}~|~ t^n)) & = -2\alpha |\lambda|, \quad (0\leq \alpha<1), \\ 
\lc(\Delta_{\lambda}(ap^{\alpha}~|~ t^n) )&= \begin{cases}
\tilde \Delta_{\lambda}^{(n)}(a;q,t) & \alpha=0, \\
\frac{\tilde C_{\lambda}^0(t^n)}{\tilde C_{\lambda}^-(q,t;q,t)}
\left(-\frac{1}{a^2q^2t^{n-1}}\right)^{|\lambda|} q^{-3 n(\lambda') } t^{5n(\lambda)}
  & 0<\alpha<1.
\end{cases}
\end{align*}

We would like to finish this subsection by making the following observation. Notice that the leading coefficients of 
these terms, only depend on whether $\alpha=0$ or $0<\alpha<1$. For general $\alpha$ it then follows that the leading coefficients  $\lc(C_{\lambda}^{\epsilon}(p^{\alpha}x) )$ and $\lc(\Delta(ap^{\alpha}~|~t^n))$ only depend on $\alpha$ through the  component of $\mathbb{R}$ which contains $\alpha$ if we cut $\mathbb{R}$ at the integers (i.e., write $\mathbb{R} = \mathbb{Z} \cup \bigcup_{n\in \mathbb{Z}} (n,n+1)$). Moreover, the leading coefficients associated to two $\alpha$'s in different components, which are related to each other by an integer shifts (i.e. either both $\alpha$'s are integers, or both are non-integers), differ by a  monomial factor (in $x$, $q$ and $t$).

\subsection{A space of functions}
A meromorphic function $f(z_i,\ldots, z_n)$ is called a $BC_n$-symmetric $p$-abelian function if it satisfies
\begin{itemize}
\item $f$ is invariant under permutations of the $z_i$;
\item $f$ is invariant under replacing any one of the $z_i$ by $1/z_i$;
\item $f$ is invariant under replacing any one of the $z_i$ by $pz_i$.
\end{itemize}
We define the space $A^{(n)}(u_0;p,q)$ as the space of all $BC_n$-symmetric $p$-abelian functions $f$ such that
\[
\prod_{i=1}^n \theta(pq z_i^{\pm 1}/u_0;q;p)_m f(\ldots,z_i,\ldots)
= \prod_{i=1}^n \frac{\Gamma(u_0 z_i^{\pm 1})}{\Gamma(u_0 q^{-m} z_i^{\pm 1})} f(\ldots,z_i,\ldots)
\]
is holomorphic for sufficiently large $m$. That is, $f$ can only have poles at the points $u_0 q^{-l} p^k$ and
$u_0^{-1} q^l p^k$ for $k\in \mathbb{Z}$ and $1\leq l\leq m$, and these poles must be simple. It should be noted that 
$A^{(1)}(u_0;p,q) = A(u_0;p,q)$ as defined in \cite{vdBRuniv}.

The definition of functions in $A^{(n)}$ does not put any conditions on what happens if we take the limit $p\to 0$, nor does it allow us to plug in values $z_i\to z_ip^{\zeta}$ and $u_0 \to u_0p^{\gamma_0}$. In order to ensure that the limit as $p\to 0$ is well-behaved and we are allowed to change variables as indicated we define $\tilde {A}^{(n)}(u_0)$ to be those functions $f\in A^{(n)}(u_0)$ such that
\begin{itemize}
\item $f(p^{\zeta_i}z_i;u_0p^{\gamma_0};q,p) \in M(z_i,u_0,q)$ for all $\zeta_i \in \mathbb{R}$ and all $\gamma_0\in \mathbb{R}$;
\item For every $M$ and $\zeta$, there exist constants $C>0$ and $\alpha\in \mathbb{R}$ such that if $\frac{1}{Mp^{\zeta}} < |z_i|,|u_0|  <Mp^{\zeta}$  and every $|p|<\frac12$ we have
\[
\left|\prod_{i=1}^n \theta(pq z_i^{\pm 1}/u_0;q;p)_m f(\ldots,z_i,\ldots)\right|  \leq C |p|^{-\alpha}.
\]
\end{itemize}
Examples of functions in $\tilde {A}^{(n)}(u_0)$ are functions in $A^{(n)}(u_0)$ which can be expressed as finite sums of products of $p$-theta functions with arguments which are monomials in the variables $z_i$, $u_0$, $q$, and $p$ (and perhaps other variables the functions may depend on). In particular the biorthogonal functions $\mathcal{R}_{\lambda}^{(n)}$ from \cite{RainsBCn} and \cite{Rainstrafo} fall in this space.

%
%
\section{Finitely supported measures}\label{secdisc}
In this section we consider the multivariate extension of the Frenkel-Turaev \cite{FT} summation. The number of terms in this sum is finite. We only consider the limit where we keep the number of terms fixed. We expect the limits where the number of terms tend to infinity to be identical to the series measures obtained as limit of the continuous measure in Proposition \ref{propsumlim}. In this fixed number of terms case taking the limit can be done trivially by exchanging limit and sum. Moreover, going to the related bilinear form can be done by just taking the limit of the functions in the finite number of points in the support of the measure. Put otherwise, this section is a completely algebraic affair, in which we are interested in obtaining the leading coefficients of some power series in $p$.

Let us recall the finitely supported bilinear form below.
\begin{definition}
Let $f\in  {A}^{(n)}(u_0)$ and $g\in A^{(n)}(u_1)$. For parameters $t_0,t_1,t_2,t_3,u_0,u_1$ such that
$t_0t_1=q^{-N}t^{1-n}$ and $t^{n-1}t_2t_3u_0u_1=pq^{N+1}$ we define the bilinear form
\[
\langle f,g\rangle_{t,u} := 
\sum_{\mu \subset N^n} f(t_0t^{n-i} q^{\mu_i}) g(t_0t^{n-i} q^{\mu_i})
w_\mu
\]
where
\[
w_\mu:=w_{\mu}(t_0,t_1;t_2,t_3,u_0,u_1)=
\frac{\Delta_{\mu}(t^{2(n-1)}t_0^2 ~|~ t^n, t^{n-1} t_0t_1,t^{n-1} t_0t_2,t^{n-1} t_0t_3, t^{n-1}t_0u_0,t^{n-1}t_0u_1 )}{
\Delta_{N^n}^0(t^{n-1}t_1/u_0 ~|~ t_1/t_0,pq/u_0t_2,pq/u_0t_3,pq/u_0u_1)}.
\]
\end{definition}
This form is normalized by $\langle 1,1\rangle_{t,u}=1$.

\begin{lemma}\label{lemweightsym}
The weights $w_\mu$ are $p$-elliptic in $t_0$, $t_1$, $t_2$, $t_3$, $u_0$, and $u_1$, as long as the balancing conditions are satisfied, which implies
\begin{align*}
w_{\mu}(pt_0,t_1/p;t_2,t_3,u_0,u_1) & = w_{\mu}(t_0,t_1;t_2,t_3,u_0,u_1), \\
w_{\mu}(t_0,t_1;pt_2,t_3,u_0,u_1/p) & = w_{\mu}(t_0,t_1;t_2,t_3,u_0,u_1), \\
w_{\mu}(t_0,t_1;t_2,pt_3,u_0,u_1/p) & = w_{\mu}(t_0,t_1;t_2,t_3,u_0,u_1), \\
w_{\mu}(t_0,t_1;t_2,t_3,pu_0,u_1/p) & = w_{\mu}(t_0,t_1;t_2,t_3,u_0,u_1),
\end{align*}
and the weights satisfy the equation
\[
w_{\mu}(t_0,t_1;t_2,t_3,u_0,u_1) = w_{\mu}(t_0p^{1/2},t_1p^{-1/2};t_2p^{1/2},t_3 p^{1/2},u_0p^{-1/2},u_1p^{-1/2}).
\]
Moreover $w_{\mu}$ is invariant under permutations of $t_2$, $t_3$, $u_0$, and $u_1$ and satisfies
\[
w_{\mu}(t_0,t_1;t_2,t_3,u_0,u_1) = w_{N^n-\mu}(t_1,t_0;t_2,t_3,u_0,u_1)
\]
\end{lemma}
\begin{proof}
This follows from direct calculations using the equations from Section \ref{secnot}.
In order to prove the $u_0\to pu_0$ and $u_1\to u_1/p$ and the final shift-by-$p^{1/2}$ equations we need to use the balancing conditions. 
It is easier to infer the $t_0\to pt_0$ and $t_1\to t_1/p$ equation from the other 4 $p$-shift equations, than to calculate it directly.

The permutation symmetry of the four final parameters follows from the equation \[C_{N^n}^0(x) = C_{N^n}^0(p q^{1-N} t^{n-1}/x)\]
(which is a consequence of the $\theta(x;p)=\theta(p/x;p)$ symmetry). 

The final equation can be shown by some complicated combinatorial arguments. Easier, however, is to note that the biorthogonal functions $\tilde R_{\lambda}^{(n)}$ (with $\lambda \subset N^n$) from \cite{RainsBCn} are biorthogonal with respect to this measure, and note that
\begin{multline*}
\sum_{\mu \subset N^n} f(t_0t^{n-i} q^{\mu_i}) g(t_0t^{n-i} q^{\mu_i}) w_\mu(t_0,t_1;t_2,t_3,u_0,u_1) \\= 
\sum_{\mu \subset N^n} f(t_1t^{n-i} q^{\mu_i}) g(t_1t^{n-i} q^{\mu_i}) w_{N^n-\mu}(t_0,t_1;t_2,t_3,u_0,u_1),
\end{multline*}
by inverting the order of summation. As the measure on a finite set is uniquely determined by a complete basis of orthogonal functions and their norms, it must follow that these measures are the same, which implies the final equation.
%
\end{proof}
In order to find the limits of the measures, it suffices to look at the limits of the 
weights. We want to consider limits $\lim_{p\to 0} w_{\mu}(t_0p^{\alpha_0},\ldots,t_3p^{\alpha_3},u_0p^{\alpha_4},u_1p^{\alpha_5})$. The above symmetries provide an action on the $\alpha$-vectors which leaves $w_{\mu}$ invariant, and so leaves the limits identical (where we possibly allow the interchange of $t_0$ and $t_1$). Thus we only want to consider vectors in some fundamental domain of that action.
\begin{lemma}
Consider the set $A$ of parameters $\alpha \in \mathbb{R}^6$ with $\alpha_0+\alpha_1=0$ and $\alpha_2+\alpha_3+\alpha_4+\alpha_5=1$. Let $G$ act on $A$ by shifts 
\begin{align*}
t_1(\alpha) &= (\alpha_0+1,\alpha_1-1, \alpha_2,\alpha_3,\alpha_4,\alpha_5) \\
t_2(\alpha) &= (\alpha_0,\alpha_1, \alpha_2+1,\alpha_3,\alpha_4,\alpha_5-1) \\
t_3(\alpha) &= (\alpha_0,\alpha_1, \alpha_2,\alpha_3+1,\alpha_4,\alpha_5-1) \\
t_4(\alpha) &= (\alpha_0,\alpha_1, \alpha_2,\alpha_3,\alpha_4+1,\alpha_5-1) \\
t_5(\alpha) &= (\alpha_0+1/2,\alpha_1-1/2, \alpha_2+1/2,\alpha_3+1/2,\alpha_4-1/2,\alpha_5-1/2)
\end{align*}
by permutations of $(\alpha_2,\alpha_3,\alpha_4,\alpha_5)$ and by permutations of $(\alpha_0,\alpha_1)$. A fundamental 
domain for this action is the polytope determined by the inequalities
\[
-\frac12 \leq \alpha_0\leq 0, \qquad \alpha_1=-\alpha_0, \qquad 
\alpha_2\leq \alpha_3\leq \alpha_4\leq \alpha_5 
, \qquad \alpha_4+\alpha_5\leq 1, \qquad
\alpha_2+\alpha_3+\alpha_4+\alpha_5=1.
\]
\end{lemma}
\begin{proof}
We have to show that from any vector we can go to this polytope. Indeed by integer shifts and permuting the last four variables we can ensure $\alpha_2\leq \alpha_3\leq \alpha_4\leq \alpha_5 \leq \alpha_2+1$. If $\alpha_4+\alpha_5 >1$ then $\alpha_2+\alpha_3<0$, so after a half integer shift and rearranging the elements in order we get $(\alpha_2,\alpha_3, \alpha_4,\alpha_5) \mapsto (\alpha_4-\frac12,\alpha_5-\frac12, \alpha_2+\frac12, \alpha_3+\frac12)$, with in particular $(\alpha_2+\frac12) + (\alpha_3+\frac12)<1$. 
 Thus we can assume $\alpha_4+\alpha_5\leq 1$. 
The equation $\alpha_5\leq \alpha_2+1$, or equivalently $\alpha_3+\alpha_4+2\alpha_5 \leq 2$ now follows from $\alpha_3\leq \alpha_4$ and $\alpha_4+\alpha_5\leq 1$ so we can omit it.
Now we use integer
shifts to get $|\alpha_0|\leq \frac12$, and if necessary interchange $\alpha_0$ and $\alpha_1$ to come into this fundamental domain. 
\end{proof}

We chose this fundamental domain, as for $\alpha$ within this fundamental domain it is straightforward to determine the valuation and leading coefficient of $w_{\mu}(p^{\alpha_r} t_r)$. Direct inspection shows that the valuation of these weights is of the form $a(\alpha)|\mu|+b(\alpha) Nn$ for some piecewise linear functions $a$ and $b$:
\begin{align}
a(\alpha)&:= -2\alpha_0 + \sum_{r\geq 1: \alpha_r+\alpha_0<0} (\alpha_0+\alpha_r) + \sum_{r\geq 1:\alpha_0>\alpha_r} (\alpha_0-\alpha_r),\\
b(\alpha)&:= \sum_{r\geq 2:\alpha_1+\alpha_r<0} (\alpha_1+\alpha_r) + \sum_{r\geq 2:\alpha_1+\alpha_r>1} (1-\alpha_1-\alpha_r). \label{eqdefbmu}
\end{align}

We want to be able to take the limit of the bilinear form if we insert the biorthogonal functions 
$\tilde R_{\lambda}^{(n)}$ from \cite{RainsBCn}. We observed in \cite{vdBRmult} that, when the parameters were specialized to the appropriate values, the valuation of 
$\tilde R_{\lambda}^{(n)}(t_0 t^{n-i} q^{\mu_i} p^{\alpha_0};t_r p^{\alpha_r})$ is independent of $\mu$.
Thus in order for the leading coefficient of the bilinear form not to reduce to a single term (either $\mu=0$ or $\mu=N^n$) in the limit, we need to have $a(\alpha)=0$ (i.e. the valuations of all weights are identical). 

Moreover we notice that $1=\tilde R_{0}^{(n)}$ is one of the functions we want to be able to insert in the bilinear form, and that we normalized the bilinear form to ensure $\langle 1, 1\rangle=1$. In particular, the valuation of the sum of all the weights must be zero, so we want the valuation of the individual weights to vanish as well. Otherwise the sum of the leading coefficients of the weights would vanish, and therefore we would not obtain a non-degenerate linear form acting on the constant functions. (Note that $b(\alpha)$ is a sum of negative terms, so it is indeed non-positive.)

The same conditions hold if we just want an interesting limiting series summation formula: a summation formula for a series of one term is rather trivial, and a series which sums to zero is also not very interesting.

It is straightforward to obtain the location in the fundamental domain where $a$ and $b$ vanish, which gives us the following theorem.
\begin{theorem}\label{thmdisclim}
Let $\alpha \in \mathbb{R}^6$ satisfy the conditions 
\begin{multline*}
-\frac12 \leq \alpha_0\leq 0, \qquad \alpha_1=-\alpha_0, \qquad 
\alpha_0\leq \alpha_r \leq \alpha_0+1, \quad (2\leq r\leq 5) 
, \qquad \alpha_r+\alpha_s\leq 1, \quad (2\leq r<s\leq 5) \\
\alpha_2+\alpha_3+\alpha_4+\alpha_5=1, 
\qquad \sum_{2\leq r\leq 5: \alpha_0+\alpha_r<0} \alpha_0+\alpha_r=2\alpha_0.
\end{multline*}
Let $t_r$ ($0\leq r\leq 5$) be parameters satisfying $t^{n-1} t_0 t_1 =q^{-N}$ and 
$t^{n-1} t_2  t_3 t_4t_5=q^{N+1}$. 
Let $f\in \tilde{A}^{(n)}(t_4)$ and $g\in \tilde{A}^{(n)}(t_5)$ 
Then $\langle f,g \rangle_{t_rp^{\alpha_r}} \in M(t_r,q,t)$ and
\begin{equation}\label{eqlimdisc}
\lim_{p\to 0} p^{-\val(f)-\val(g)}\langle f,g \rangle_{t_rp^{\alpha_r}} 
=
\sum_{\mu \subset N^n} 
\lc(f)(t_0t^{n-i} q^{\mu_i}) \lc(g)(t_0t^{n-i} q^{\mu_i}) w_{\mu,\alpha}(t_r),
\end{equation}
with $\lc(f)=\lc(f(zp^{\alpha_0};t_rp^{\alpha_r}))$ and similarly for $\lc(g)$ and the valuations, and where the weights $w_{\mu,\alpha}$ satisfy
\[
\sum_{\mu \subset N^n} w_{\mu,\alpha}=1.
\]
%
%
Here the weights $w_{\mu,\alpha}$ are given as (using the notation $t_4=u_0$ and $t_5=u_1$)
\begin{itemize}
\item If $\alpha_0=0$ 
\begin{align*}
{w}_{\mu,\alpha} (t_0, t_1;t_2,t_3,t_4,t_5) 
 & = 
 \frac{\tilde C_{2\mu^2}^0(qt^{2(n-1)} t_0^2) \tilde C_{\mu}^0(t^n,q^{-N})(\frac{1}{qt^{4(n-1)}t_0^3t_1})^{|\mu|}  q^{-2n(\mu')}t^{4n(\mu)} }{\tilde C_{\mu}^0(qt^{n-2} t_0^2,qt^{n-1}t_0/t_1) \tilde C_{\mu}^-(q,t) \tilde C_{\mu}^+(t^{2(n-1)} t_0^2,qt^{2n-1} t_0^2)\tilde C_{N^n}^0(t_1/t_0)}
\\&\times 
\prod_{2\leq r\leq 5:\alpha_r=0}
\frac{\tilde C_{\mu}^0(t^{n-1}t_0t_r) \tilde C_{N^n}^0(t^{n-1}t_1t_r)}{\tilde C_\mu^0(qt^{n-1}t_0/t_r)}
(-qt^{n-1}t_0/t_r)^{|\mu|} q^{n(\mu')} t^{-n(\mu)}
\\ & \times \prod_{2\leq r\leq 5:\alpha_r=1}
\frac{\tilde C_{\mu}^0(t^{n-1}t_0t_r) \tilde C_{N^n}^0(\frac{q^{1-N}}{t_1t_r})}{\tilde C_\mu^0(qt^{n-1}t_0/t_r)}
(-t^{n-1}t_0t_r)^{-|\mu|}  q^{-n(\mu')} t^{n(\mu)}
\\ & \times
\prod_{2\leq r<s\leq 5: \alpha_r+\alpha_s=1} \frac{1}{\tilde C_{N^n}^0(q/t_rt_s)} 
\end{align*}

\item If $-1/2<\alpha_0<0$ 
\begin{align*}
 {w}_{\mu,\alpha} (t_0, t_1;t_2,t_3,t_4,t_5) 
& =  \frac{\tilde C_{\mu}^0(t^n,q^{-N}) (t_0^2t^{2(n-1)})^{-|\mu|} q^{-2n(\mu')} t^{4n(\mu)}}{\tilde C_{\mu}^-(q,t)  }
\\ & \qquad \times
 \prod_{2\leq r\leq 5:\alpha_r=\alpha_0}
\frac{ (qt^{2(n-1)}t_0^2)^{|\mu|} q^{2n(\mu')} t^{-2n(\mu)}}{ \tilde C_{\mu}^0(qt^{n-1}t_0/t_r)}
\tilde C_{N^n}^0(t^{n-1}t_1t_r)
\\ & \qquad \times \prod_{2\leq r\leq 5:\alpha_0<\alpha_r<-\alpha_0}   (-t^{n-1}t_0t_r)^{|\mu|} q^{n(\mu')} t^{-n(\mu)}
\prod_{2\leq r\leq 5:\alpha_r=-\alpha_0} \tilde C_{\mu}^0(t^{n-1}t_0t_r)
\\ & \qquad \times \prod_{2\leq r\leq 5:\alpha_r=1+\alpha_0}\frac{
\tilde C_{N^n}^0(qt^{n-1}t_0/t_r)}{\tilde C_{\mu}^0(qt^{n-1}t_0/t_r)}
\prod_{2\leq r<s\leq 5: \alpha_r+\alpha_s=1} \frac{1}{\tilde C_{N^n}^0(q/t_rt_s)}.
\end{align*}

\item If $\alpha_0=-1/2$ 
\begin{align*}
 {w}_{\mu,\alpha} (t_0, t_1;t_2,t_3,t_4,t_5)
& = 
\frac{\tilde C_{2\mu^2}^0(qt^{2(n-1)}t_0^2)\tilde C_{\mu}^0(t^n,q^{-N})(q  t_0/t_1)^{|\mu|}  q^{2n(\mu')}}{\tilde C_{\mu}^0(qt^{n-2}t_0^2,qt^{n-1}\frac{t_0}{t_1}) \tilde C_{\mu}^-(q,t) \tilde C_{\mu}^+(t^{2(n-1)}t_0^2, qt^{n-2}t_0^2)\tilde C_{N^n}^0(q^{1-N}t^{n-1}\frac{t_0}{t_1})} 
\\ & \qquad \times \prod_{2\leq r\leq 5: \alpha_r=-1/2}
 \frac{\tilde C_{\mu}^0(t^{n-1}t_0t_r) \tilde C_{N^n}^0(t^{n-1}t_1t_r) (-\frac{qt^{n-1}t_0}{t_r})^{|\mu|} q^{n(\mu')}t^{-n(\mu)}}{
\tilde C_{\mu}^0(qt^{n-1}t_0/t_r) } 
\\ & \qquad \times \prod_{2\leq r\leq 5: \alpha_r=1/2} \frac{\tilde C_{\mu}^0(t^{n-1}t_0t_r) \tilde C_{N^n}^0(\frac{q^{1-N}}{t_1t_r})}{\tilde C_{\mu}^0(qt^{n-1}t_0/t_r)} 
(-t^{n-1}t_0t_r)^{-|\mu|} q^{-n(\mu')} t^{n(\mu)} 
\\ & \qquad \times
\prod_{\substack{2\leq r<s \leq 5 \\ \alpha_r+\alpha_s=1}} \frac{1}{\tilde C_{N^n}^0(q/t_rt_s)}.
\end{align*}
\end{itemize}
\end{theorem}
Note that we cannot guarantee that $\val(\langle f,g \rangle_{t_r}) = \val(f)+\val(g)$ (for example if $f$ and $g$ are orthogonal to each other), so we had to write the main statement in this theorem as a limit. However, for generic $f$ and $g$ the equation $\val(\langle f,g \rangle_{t_r}) = \val(f)+\val(g)$ holds.
\begin{proof}
By direct computation we observe that $\val(w_\mu(t_rp^{\alpha_r}))=0$, and, as the inner product is defined as just a finite sum, we can interchange limit and sum. The limits for $f(t_0 t^{n-i} q^{\mu_i} p^{\alpha_0})$ is obvious as a member of $M(t_r,q,t)$, likewise for $g$. The explicit expressions for $w_{\mu,\alpha}$ are obtained by direct calculation.
\end{proof}

In the case $\alpha=(0,0,0,0,\frac12,\frac12)$, or $\alpha=(-\frac12,\frac12,\frac12,\frac12,0,0)$ (which give the same measures by the symmetry relations from Lemma \ref{lemweightsym}), the measure becomes the discrete measure for 
the multivariate $q$-Racah polynomials, \cite[(3.10)]{SvDmultRacah}.

\section{Series with more parameters} We are also interested in the extension of the limits above to series with more parameters (thus in the notation of \cite{Rainstrafo}, with higher $m$). The series with two extra parameters ($m=1$) satisfies a transformation identity conjectured by Warnaar \cite[Conjecture 6.1]{Warnaar} and proved by the second author \cite[Theorem 4.9]{RainsBCn}. In this more general case we lack an explicit evaluation formula for the series. As a consequence we cannot easily determine the valuation of the series. This leads to two complications. First of all we can no longer use the heuristic that the valuation of summands must equal the valuation of the complete series to exclude uninteresting cases. Secondly we must now worry that the limiting series we obtain by simply interchanging limit and sum vanishes identically. To simplify notation, we will not consider bilinear forms in this section, but only the plain series, but there are no other complications to turn this series into a bilinear form and take the limit thereof than those which have been discussed in Section \ref{secdisc}.

The series we consider is 
\begin{definition} For parameters $t_r \in (\mathbb{C}^*)^{2m+6}$ such that
$t_0t_1=q^{-N}t^{1-n}$ and $t^{2(n-1)}\prod_{r=0}^{2m+5} t_r=(pq)^{m+1}$ we consider  the series 
\[
\sum_{\mu \subset N^n} w_\mu(t_r)
\]
where
\[
w_\mu(t_r) :=
\Delta_{\mu}(t^{2(n-1)}t_0^2 ~|~ t^n, t^{n-1} t_0t_1,t^{n-1} t_0t_2, \ldots, t^{n-1} t_0t_{2m+5}).
\]
\end{definition}
Notice that the $m=0$ case is the series $\langle 1,1\rangle$ from the previous section up to scaling.
As before we can easily obtain some simple symmetries.
\begin{lemma}\label{lemhighermsym}
The weights $w_\mu$ are $p$-elliptic in the $t_r$ as long as the balancing conditions are satisfied. Moreover they satisfy 
\[
w_{\mu}(t_0,t_1;t_2,\ldots,t_{2m+5}) = w_{\mu}(t_0p^{1/2},t_1p^{-1/2};t_2p^{1/2},\ldots, t_{m+3} p^{1/2},t_{m+4}p^{-1/2}, \ldots, t_{2m+5}p^{-1/2}).
\]
Finally $w_{\mu}$ is invariant under permutations of $(t_2,\ldots,t_{2m+5})$ and satisfies
\[
w_{\mu}(t_0,t_1;t_2,\ldots,t_{2m+5}) = w_{N^n-\mu}(t_1,t_0;t_2,\ldots,t_{2m+5}) \frac{C_{N^n}^0(t_1/t_0)}{C_{N^n}^0(t_0/t_1)} \prod_{r=2}^{2m+5} \frac{C_{N^n}^0(t^{n-1}t_0t_r)}{C_{N^n}^0(t^{n-1}t_1t_r)}.
\]
\end{lemma}
\begin{proof}
The proof is the same as for Lemma \ref{lemweightsym}, except for the final equation. We cannot use the previous argument of the weights on both sides of the equation determining the same measure as we do not have biorthogonal functions. Thus it seems that we are forced to use some combinatorial argument to equate the two sides of products of theta functions. However we can circumvent most of the complications by observing that the equation holds for $m=0$, and using the equation
\[
C_{\mu}^0(x) C_{N^n-\mu}^0(pq^{1-N}t^{n-1}/x) = C_{N^n}^0(x)
\]
to adjust the equation for the extra $C_{\mu}^0$ terms.
\end{proof}
We now want to consider $w_{\mu}(t_rp^{\alpha_r})$. The above symmetries provide a group action on the $\alpha_r$ which preserves $w_{\mu}$. We could not find a pretty fundamental domain for the action which includes the $t_0\leftrightarrow t_1$ interchange, so we just consider the following
\begin{lemma}\label{lemfd2}
Consider the set $A$ of parameters $\alpha \in \mathbb{R}^{2m+6}$ with $\alpha_0+\alpha_1=0$ and $\sum_{r= 2}^{2m+5} \alpha_r=m+1$. Let $G$ act on $A$ by integer shifts preserving the balancing conditions, by the half integer shift
\begin{align*}
\alpha \to (\alpha_0-\frac12,\alpha_1+\frac12;\alpha_2+\frac12,\ldots,\alpha_{m+3}+\frac12,\alpha_{m+4}-\frac12,\ldots,\alpha_{2m+5}-\frac12),
\end{align*}
and by permutations of $(\alpha_2,\ldots,\alpha_{2m+5})$. A fundamental 
domain for this action is the polytope determined by the inequalities
\[
-\frac12 \leq \alpha_0\leq 0, \qquad \alpha_1=-\alpha_0, \qquad 
\alpha_2\leq \alpha_3\leq \cdots \leq \alpha_{2m+5}\leq \alpha_2+1, \qquad
\sum_{r=2}^{2m+5} \alpha_r=1.
\]
In this fundamental domain we have $\alpha_2\geq -\frac12$ and $\alpha_{2m+5} \leq \frac32$.
\end{lemma}
\begin{proof}
The verification that this is a fundamental domain is as before, except that we start with half-integer shifting $\alpha_0$ until it is in the given interval.

For the verification of the bound on $\alpha_2$ we notice that
\[
m+1 =  \alpha_2 + \sum_{r=3}^{2m+5} \alpha_r \leq 
\alpha_2 + \sum_{r=3}^{2m+5} (\alpha_2 +1)  = (2m+4)\alpha_2 + (2m+3),
\]
which simplifies to $\alpha_2 \geq -\frac12$. The equation for $\alpha_{2m+5}$ is obtained similarly.
\end{proof}
In this fundamental domain the valuation of the summands $w_{\mu}$ is determined to be
\begin{multline*}
\val(w_{\mu}(t_rp^{\alpha_r}))  =  |\mu|\Big( -2\alpha_0+ \sum_{r\geq 2:\alpha_r\leq \alpha_0} 2\alpha_0 + \sum_{r\geq 2:\alpha_0<\alpha_r<-\alpha_0} (\alpha_0+\alpha_r)  \\+ \sum_{r\geq 2:1+\alpha_0<\alpha_r<1-\alpha_0} (\alpha_r-1-\alpha_0)
 + \sum_{r\geq 2:\alpha_r\geq 1-\alpha_0} (-2\alpha_0) \Big).
\end{multline*}
As before, we insist on this valuation to be equal for all values of $\mu$, thus we conclude that the term between brackets must vanish. Let us determine when this happens. 
\begin{lemma}\label{lemzero}
Let $\alpha_r$ be in the fundamental domain from Lemma \ref{lemfd2}. Then 
\begin{multline*}
 -2\alpha_0+ \sum_{r\geq 2:\alpha_r\leq \alpha_0} 2\alpha_0 + \sum_{r\geq 2:\alpha_0<\alpha_r<-\alpha_0} (\alpha_0+\alpha_r)  \\+ \sum_{r\geq 2:1+\alpha_0<\alpha_r<1-\alpha_0} (\alpha_r-1-\alpha_0)
 + \sum_{r\geq 2:\alpha_r\geq 1-\alpha_0} (-2\alpha_0)=0
\end{multline*}
if and only if any of the following holds
\begin{itemize}
\item $\alpha_0=0$ or $\alpha_0=-\frac12$;
\item $\alpha_2\leq \alpha_0$ and $-\alpha_0\leq \alpha_r \leq 1+\alpha_0$ for $r\geq 3$;
\item $\alpha_0<\alpha_r<1-\alpha_0$ for $r\geq 2$, $2\leq A:=\#\{r:\alpha_r<-\alpha_0\}\leq m+3$, $B:=\#\{r:\alpha_r>1+\alpha_0\}\leq m+1$ and 
\begin{equation}\label{eqwhenzero}
\sum_{r:-\alpha_0\leq \alpha_r\leq 1+\alpha_0} \alpha_r= (m+1-B) + (A-2-B) \alpha_0.
\end{equation}
For each choice of $A$, $B$ and $\alpha_0$ there exist values of $\alpha_r$ ($r\geq 2$) which solve this equation. If $A=m+3$ then $\alpha_r=1+\alpha_0$ for all $r$ such that $-\alpha_0\leq \alpha_r\leq 1+\alpha_0$ and if $B=m+1$ then $\alpha_r=-\alpha_0$ for all $r$ such that $-\alpha_0\leq \alpha_r\leq 1+\alpha_0$\footnote{If $A=m+3$ and $B=m+1$ then there are no $r$ such that $-\alpha_0\leq \alpha_r\leq 1+\alpha_0$, so this does not lead to a contradiction.}.
\end{itemize}
\end{lemma}
\begin{proof}
If $\alpha_0=0$ then all terms vanish, and thus we get identically 0 as desired. This would imply by the $p^{1/2}$-shift equation that the valuation also vanishes identically if $\alpha_0=-\frac12$, and this is indeed seen to be true by using the balancing condition (essentially the left hand side becomes $1 +\sum_{r\geq 2} (\alpha_r-\frac12)$). 

For $-\frac12<\alpha_0<0$ we see that $-2\alpha_0$ and the terms for $\alpha_r>1+\alpha_0$ are (strictly) positive, while the terms for $\alpha_r<-\alpha_0$ are negative. But if $\alpha_{2m+5}\geq 1-\alpha_0$ then all $\alpha_r\geq -\alpha_0$ (for $r\geq 2$), so we have a sum of positive terms, which can never vanish. So we see that $\alpha_{2m+5}<1-\alpha_0$. If $\alpha_2\leq \alpha_0$ we see that $\alpha_r\leq 1+\alpha_0$ for $r\geq 2$, so the term between brackets simplifies to \[
\sum_{r\geq 3:\alpha_\leq \alpha_0} 2\alpha_0 + \sum_{r\geq 3:\alpha_0<\alpha_r<-\alpha_0} (\alpha_0+\alpha_r),\] which is a sum of all negative terms, so this can vanish only if $\alpha_r \geq -\alpha_0$ for $r\geq 3$. Otherwise, we can assume $\alpha_0<\alpha_2$ and $\alpha_{2m+5} < 1-\alpha_0$, and see that we have a solution as long as 
$\alpha_{2m+5} \leq \alpha_2+1$, and the equation reduces to 
\begin{multline*}
0 = -2\alpha_0 + A\alpha_0 + \sum_{r\geq 2: \alpha_r<-\alpha_0} \alpha_r 
-B(1+\alpha_0)  + \sum_{r\geq 2: \alpha_r>1+\alpha_0} \alpha_r 
 \\ = -B + (A-2-B) \alpha_0 + \left( m+1- \sum_{r\geq 2: -\alpha_0\leq \alpha_r\leq 1+\alpha_0} \alpha_r \right)
\end{multline*}
which simplifies to the given equation. Now we observe that we can only get a solution if 
\[
-\alpha_0 \leq 
\frac{\sum_{r\geq 2: -\alpha_0 \leq \alpha_r \leq 1+\alpha_0} \alpha_r}{\#\{r\geq 2:-\alpha_0\leq \alpha_r \leq 1+\alpha_0\}} 
\leq 1+\alpha_0.
\]
The numerator of the quotient is expressed in terms of $A$, $B$ and $\alpha_0$ above, while the denominator is seen to equal $2m+4-A-B$. Simplifying the resulting equation gives the upper bounds on $A$ and $B$ (one must separately consider the case $2m+4-A-B=0$). The lower bound on $A$ follows from the fact that the left hand side would otherwise be positive. Given $A$, $B$ and $\alpha_0$ we can now set $A$ values of $\alpha_r$ equal to $-\alpha_0 + \frac{(2+c)}{A} \alpha_0$, $2m+4-A-B$ values of 
$\alpha_r$ equal to $\frac{(m+1-B) + (A-2-B) \alpha_0}{2m+4-A-B}$ and 
$B$ values of $\alpha_r$ equal to $1+\alpha_0 - \frac{c}{B}\alpha_0$ where $c:=\frac{2B(A-1)}{A+B}$. Here $c$ was chosen such that $\left(-\alpha_0 + \frac{(2+c)}{A}\alpha_0\right) +1 = 1+\alpha_0 - \frac{c}{B}\alpha_0$.
\end{proof}
The same calculations that show that $A\leq m+3$ also show that $\#\{r\geq 2:\alpha_r\leq -\alpha_0\} \leq m+3$ (i.e. we include the ones at $-\alpha_0$), and likewise we obtain that 
$\#\{r\geq 2: \alpha_r\geq 1+\alpha_0\}\leq m+1$. 

We can now give the following theorem
\begin{theorem}
Let $\alpha_r \in (\mathbb{C}^*)^{2m+6}$ be in the fundamental domain of Lemma \ref{lemfd2} and satisfy the condition from Lemma \ref{lemzero}. Moreover assume $t_0t_1=q^{-N} t^{1-n}$ and 
$\prod_{r=0}^{2m+5} t_r = q^{m+1}$.
Then we obtain 
\begin{equation}\label{eqlcswitch}
\lc\left( \sum_{\mu \subset N^n} w_{\mu}(t_r p^{\alpha_r})\right) 
= \sum_{\mu \subset N^n} \lc(w_\mu(t_rp^{\alpha_r})),
\end{equation}
unless (using $A$ and $B$ as before)
\begin{itemize}
\item $\alpha_0=0$, $A=1$, $B=0$ and $\alpha_r\not \in \mathbb{Z}$ for $r\geq 2$;
\item $-\frac12<\alpha_0<0$,  $\alpha_2<\alpha_0$ and $-\alpha_0<\alpha_r<1+\alpha_0$ for all $r\geq 3$;
\item $-\frac12<\alpha_0<0$, $A=m+3$ and $B=m+1$;
\item $\alpha_0=-\frac12$, $\alpha_2>-\frac12$, $A=m+3$ and $B=m+1$.
\end{itemize}
in which cases the right hand side vanishes. Here $\lc(w_\mu(t_rp^{\alpha_r}))$ is given by the equations
\begin{itemize}
\item If $\alpha_0=0$ then 
\begin{align*}
\lc&(w_\mu(t_rp^{\alpha_r})) = 
\frac{\tilde C_{2\mu^2}^0(qt^{2(n-1)} t_0^2) \tilde C_{\mu}^0(t^n,q^{-N})(q^{N-1}t^{-3(n-1)}t_0^{-2})^{|\mu|} q^{-2n(\mu')} t^{4n(\mu)}}{\tilde C_{\mu}^0(qt^{n-2}t_0^2,q^{N+1}t^{2(n-1)} t_0^2) \tilde C_{\mu}^-(q,t) \tilde C_{\mu}^+(t^{2(n-1)}t_0^2,qt^{2n-1}t_0^2)} 
\\ & \times \left(( qt^{2(n-1)}t_0^2)^{|\mu|} q^{2n(\mu')} t^{-2n(\mu)}\right)^{A-B}
\prod_{r\geq 2:\alpha_r=0} \left(\frac{\tilde C_{\mu}^0(t^{n-1}t_0t_r)}{\tilde C_{\mu}^0(qt^{n-1}t_0/t_r)} \left( -\frac{qt^{n-1}t_0}{t_r}\right)^{|\mu|} q^{n(\mu')}t^{-n(\mu)} \right) 
\\ & \times \prod_{r\geq 2:\alpha_r=1} \left(\frac{\tilde C_{\mu}^0(t^{n-1}t_0t_r)}{\tilde C_{\mu}^0(qt^{n-1}t_0/t_r)} \left( -\frac{1}{t^{n-1}t_0t_r}\right)^{|\mu|} q^{-n(\mu')}t^{n(\mu)} \right).
\end{align*}
\item If $-\frac12<\alpha_0<0$ and $\alpha_2 \leq \alpha_0$ then 
\begin{align*}
\lc(w_\mu(t_rp^{\alpha_r})) &= 
\frac{\tilde C_{\mu}^0(t^n,q^{-N})}{\tilde C_{\mu}^-(q,t)} t^{2n(\mu)} q^{|\mu|}
\left(\frac{1}{\tilde C_{\mu}^0(qt^{n-1}t_0/t_2)} \right)^{1_{\{\alpha_2=\alpha_0\}}}
\frac{\prod_{r\geq 3:\alpha_r=-\alpha_0} \tilde C_{\mu}^0(t^{n-1}t_0t_r)}{
\prod_{r\geq 3:\alpha_r=1+\alpha_0} \tilde C_{\mu}^0(qt^{n-1}t_0/t_r)}.
\end{align*}
\item If $-\frac12<\alpha_0<0$ and $\alpha_2 >\alpha_0$ then 
\begin{align*}
\lc(w_\mu(t_rp^{\alpha_r})) &= 
\frac{\tilde C_{\mu}^0(t^n,q^{-N})}{\tilde C_{\mu}^-(q,t)} (-t^{n-1}t_0)^{(A-B-2)|\mu|} q^{-B|\mu|}
q^{(A-B-2)n(\mu')} t^{(B-A+4)n(\mu)} 
\\ & \qquad \times \frac{\prod_{r\geq 3:\alpha_r=-\alpha_0} \tilde C_{\mu}^0(t^{n-1}t_0t_r)}{
\prod_{r\geq 3:\alpha_r=1+\alpha_0} \tilde C_{\mu}^0(qt^{n-1}t_0/t_r)}  z^{|\mu|}.
\end{align*}
where $z=\prod_{r\geq 2:\alpha_r<-\alpha_0} t_r \prod_{r\geq 2:\alpha_r>1+\alpha_0} t_r$

\item If $\alpha_0=-\frac12$ and $\alpha_2=-\frac12$ (so $\alpha_r=\frac12$ for $r\geq 3$) then 
\begin{align*}
\lc(w_\mu(t_rp^{\alpha_r})) & = 
\frac{\tilde C_{2\mu^2}^0(qt^{2(n-1)} t_0^2) \tilde C_{\mu}^0(t^n,q^{-N})}{\tilde C_{\mu}^0(qt^{n-2}t_0^2,q^{N+1} t^{2(n-1)} t_0^2) \tilde C_{\mu}^-(q,t) \tilde C_{\mu}^+(t^{2(n-1)}t_0^2,qt^{2n-1}t_0^2)} 
q^{|\mu|} t^{2n(\mu)}  
\\& \qquad \times 
\prod_{r=2}^{2m+5}  \left( \frac{\tilde C_{\mu}^0(t^{n-1}t_0t_r)}{\tilde C_{\mu}^0(qt^{n-1}t_0/t_r)} \right).
\end{align*}
\item If $\alpha_0=-\frac12$ and $\alpha_2>-\frac12$ then 
\begin{align*}
\lc(w_\mu(t_rp^{\alpha_r})) &= 
\frac{\tilde C_{2\mu^2}^0(qt^{2(n-1)} t_0^2) \tilde C_{\mu}^0(t^n,q^{-N})}{\tilde C_{\mu}^0(qt^{n-2}t_0^2,q^{N+1} t^{2(n-1)} t_0^2) \tilde C_{\mu}^-(q,t) \tilde C_{\mu}^+(t^{2(n-1)}t_0^2,qt^{2n-1}t_0^2)} 
q^{-B |\mu|} (-t^{n-1}t_0)^{(A-B-2)|\mu|} 
\\ & \qquad \times q^{(A-B-2)n(\mu')}t^{(B+4-A)n(\mu)}  \prod_{r:\alpha_r=\frac12 } \left( \frac{\tilde C_{\mu}^0(t^{n-1}t_0t_r)}{\tilde C_{\mu}^0(qt^{n-1}t_0/t_r)} \right)z^{|\mu|},
\end{align*}
where $z=\prod_{r\geq 2:\alpha_r\neq \frac12} t_r$
\end{itemize}

\end{theorem}
\begin{proof}
Obtaining the leading coefficients of the individual summands is a straightforward calculation. Moreover the result of Lemma \ref{lemzero} indicates that we only consider cases in which the valuation of all summands is equal to the valuation of $w_0=1$. Thus it suffices to show that in these cases the right hand side of \eqref{eqlcswitch} does not vanish, and that it does vanish in the other cases. 

In the cases in which we claim the right hand side of \eqref{eqlcswitch} vanishes, the series is identical as a series obtained in the $m=0$ case for one of the cases $\alpha=(0,0,-\frac12,\frac12,\frac12,\frac12)$,  
$\alpha=(-\frac14,\frac14,-\frac12,\frac12,\frac12,\frac12)$, $\alpha= (-\frac14,\frac14,0,0,0,1)$, or 
$\alpha=(-\frac12,\frac12,0,0,0,1)$. These are cases for which  which we determined in the previous section that the valuation of the series is more than the valuation of the individual summands (that is $b(\mu)$ of \eqref{eqdefbmu} is strictly negative). In particular we have already seen that these series vanish.

For the other cases we note that Corollary \ref{corsumequalval} implies that if we can modify our vector $\alpha$ by an arbitrarily small amount to make the equation from Lemma \ref{lemzero} fail, we would have a further limit to a series consisting of just 1 term (either the term $\mu=0$ or the term $\mu=N^n$), which therefore cannot vanish. In particular, in this case, the series corresponding to our original $\alpha$ also cannot vanish. 

Let us first consider the case $-\frac12 <\alpha_0<0$. If we increase one $\alpha_r$ ($r\geq 2$) we must decrease another one in order to preserve the balancing condition. Thus we see that unless the derivative in $\alpha_r$ of the left hand side of the equation in Lemma \ref{lemzero} is equal for all $r\geq 2$, our series does not vanish. The derivative to $\alpha_r$ is 1 if $\alpha_0<\alpha_r<-\alpha_0$ or integer shifts thereof, and 0 if $-\alpha_0<\alpha_r<1+\alpha_0$ and integer shifts thereof, and does not exist (left and right derivatives differ) if $\alpha_r \in \pm \alpha_0 + \mathbb{Z}$. Thus the only cases where all left and right derivatives are identical are when either all $\alpha_r$ are in the union of open intervals $(-\alpha_0,1+\alpha_0)+\mathbb{Z}$ or if they are all in $(\alpha_0,-\alpha_0)+\mathbb{Z}$. In our possible solutions the first case is only the case with $\alpha_2<\alpha_0$ and $\alpha_r \in (-\alpha_0,1+\alpha_0)$ for $r\geq 3$, for which we have already seen that the series vanishes. For the second case we notice that our equation reduces to 
\eqref{eqwhenzero} with the left hand side equal to zero and $A+B=2m+4$. The derivative to $\alpha_0$ of this equation is clearly non-zero precisely when $A-2-B\neq 0$, so in those cases we also cannot vanish, whereas the case $A-2-B=0$ corresponds to $A=m+3$ and $B=m+1$, of which we have seen that the series vanishes.

The cases $\alpha_0=0$ and $\alpha_0=-\frac12$ are similar (and indeed related by a half-integer shift), so we only have to consider the case $\alpha_0=0$. If there exists an $\alpha_r\in \mathbb{Z}$ we can change $\alpha_0$ and $\alpha_r$ simultaneously, while keeping $\alpha_0-\alpha_r$ fixed, and can thus take a limit to a case with $-\frac12 <\alpha_0<0$ and an $\alpha_r \in \pm \alpha_0 +\mathbb{Z}$, of which we know the series exists. So we only need to consider the cases where no $\alpha_r \in \mathbb{Z}$. In that case the derivative to $\alpha_0$ equals $2(-1+A-B)$, so we see that the series does not vanish unless $A=B+1$. Observe that our fundamental domain is such that, if $\alpha_0=0$, either $A=0$ or $B=0$, hence we see that the only vanishing case is $A=1$ and $B=0$ as claimed.
\end{proof}

\section{Absolutely continuous measures}\label{secac}
For the continuous measures we cannot simply refer to our algebraic framework, as we cannot necessarily expand integrals of power series in $p$ in such power series. Thus in this section we mostly deal with ordinary limits.
Generic parameters in this subsection are parameters $t$, $t_r$ satisfying a balancing condition (usually either  $t^{2(n-1)} \prod_{r=0}^5 t_r = pq$ or $t^{2(n-1)} \prod_{r=0}^5 t_r = q$) such that $t_rt_s \not \in p^{\mathbb{Z}_{\leq 0}} q^{\mathbb{Z}_{\leq 0}} t^{\mathbb{Z}_{\leq 0}}$ for $0\leq r,s\leq 3$ and $t_rt_4,t_rt_5 \not \in p^{\mathbb{Z}_{\leq 0} }q^{\mathbb{Z}}t^{\mathbb{Z}_{\leq 0}}$ for $0\leq r\leq 5$. Moreover we assume $|p|,|q|,|t|<1$. 

We will assume $f\in \tilde {A}^{(n)}(t_4)$ and $g\in \tilde{A}^{(n)}(t_5)$ throughout this section, and it becomes convenient to fix $m_f$ and $m_g$ such that 
\[
\hat f(z) := \prod_{i=1}^n \frac{\Gamma(t_4z_i^{\pm 1})}{\Gamma(t_4q^{-m_f} z_i^{\pm 1})} f(z), \qquad 
\hat g(z):=\prod_{i=1}^n \frac{\Gamma(t_5z_i^{\pm 1})}{\Gamma(t_5q^{-m_g} z_i^{\pm 1})} g(z_i),
\]
are holomorphic and define $\tilde t_r=t_r$ if $r=0,1,2,3$, $\tilde t_4=t_4 q^{-m_f}$ and $\tilde t_5 = t_5 q^{-m_g}$. In some cases we will moreover use 
$m_r:=\log_q( t_r/\tilde t_r)$, so $m_r=0$ for $r=0,1,2,3$, $m_4=m_f$ and $m_5=m_g$.
The first consequence is that we can immediately observe the analytic properties of $\lc(f)=\lc(f(zp^{\zeta};t_rp^{\alpha_r})$, as $\lc(\hat f)$ is 
a Laurent polynomial in the $z_i$ (as a rational function which is holomorphic on 
$\left( \mathbb{C}^* \right)^{n}$) and 
\begin{equation}\label{eqlcfinlchatf}
\lc(\hat f) = \lc(f) \prod_{i=1}^n 
(\tilde u_0 z_i;q)_{m_f}^{1_{\alpha_4+\zeta\in \mathbb{Z}}}
(\tilde u_0/z_i;q)_{m_f}^{1_{\alpha_4-\zeta\in \mathbb{Z}}}
\left( \left(-\frac{q}{u_0 z_i} \right)^{m_f} q^{\binom{m_f}{2}} \right)^{\lfloor \alpha_4+\zeta\rfloor} 
\left( \left(-\frac{qz_i}{ u_0} \right)^{m_f} q^{\binom{m_f}{2}} \right)^{\lfloor \alpha_4-\zeta\rfloor}.
\end{equation}

Let us now recall the definition of the continuous bilinear form from \cite{Rainstrafo}.
\begin{definition}
For generic parameters $t\in \mathbb{C}^6$ satisfying $t^{2(n-1)} \prod_{r=0}^5 t_r = pq$ 
we define the bilinear form on $f\in A^{(n)}(t_4)$ and $g\in A^{(n)}(t_5)$ as 
\begin{multline*}
\langle f,g\rangle_{t_0,t_1,t_2,t_3,t_4,t_5:q,t;p} = 
\frac{(q;q)^n (p;p)^n \Gamma(t;p,q)^n}{2^n n! \prod_{j=1}^n\Gamma(t^j;p,q) \prod_{0\leq r<s\leq 5} \Gamma(t^{n-j} t_rt_s;p,q)} \\ \times 
\int_{C^n} f(\cdots,z_i,\cdots) g(\cdots,z_i,\cdots) \prod_{1\leq j<k\leq n} \frac{\Gamma(tz_j^{\pm 1} z_k^{\pm 1};p,q)}{\Gamma(z_j^{\pm 1}z_k^{\pm 1};p,q)}
\prod_{j=1}^n \frac{\prod_{r=0}^5 \Gamma(t_rz_j^{\pm 1};p,q)}{\Gamma(z_j^{\pm 2};p,q)} \frac{dz_j}{2\pi i z_j},
\end{multline*}
where the integration contour $C=C^{-1}$ contains all points of the form $p^i q^j \tilde t_r$ ($i,j\geq 0$), excludes their reciprocals, and contains $p^iq^jtC$
($i,j\geq 0$) (i.e. its own image when multiplied by some number)\footnote{To be precise, $C$ should be a chain representing the described homology class.}.

If $|\tilde t_r|<1$ for all $r$, then the integration contour can be taken as the unit circle.
\end{definition}
Note that the specific choice of $m_f$ and $m_g$ does not change the value of the integral, as the only difference between the specification of the integration contours is whether we explicitly force them to contain some points which are no poles of the integrand.

It is convenient to rewrite the integrand using $\hat f$ and $\hat g$ as 
\begin{multline}\label{eqacip2}
\langle f,g\rangle_{t_0,t_1,t_2,t_3,t_4,t_5:q,t;p} = 
\frac{(q;q)^n (p;p)^n \Gamma(t;p,q)^n}{2^n n! \prod_{j=1}^n\Gamma(t^j;p,q) \prod_{0\leq r<s\leq 5} \Gamma(t^{n-j} t_rt_s;p,q)} \\ \times 
\int_{C^n} \hat f(z_i) \hat g(z_i) \prod_{1\leq j<k\leq n} \frac{\Gamma(tz_j^{\pm 1} z_k^{\pm 1};p,q)}{\Gamma(z_j^{\pm 1}z_k^{\pm 1};p,q)}
\prod_{j=1}^n \frac{\prod_{r=0}^5 \Gamma(\tilde t_rz_j^{\pm 1};p,q)}{\Gamma(z_j^{\pm 2};p,q)} \frac{dz_j}{2\pi i z_j},
\end{multline}

Now we want to obtain several limits of these bilinear forms. The easiest way to obtain such a limit is if there exists some valuation $\zeta$, such that for given functions $f$ and $g$ there exists a fixed contour (for small enough $p$) for the integral expression of $\langle f,g\rangle$ after rescaling $z \to p^{\zeta}z$ in the integral. In those cases we can just replace limit and integral.
The easiest example is the following (where $\zeta=0$).
\begin{proposition}\label{propPI}
Choose generic parameters satisfying $t^{2(n-1)} \prod_r t_r=q$.
Let $\alpha \in \mathbb{R}^6$, $\sum_{r=0}^5 \alpha_r=1$ and $\alpha_r\geq 0$ for $0\leq r\leq 5$. 


We now have the limit
\begin{align*}
\lim_{p\to 0} p^{-\val(f)-\val(g)} & \langle f,g\rangle_{t_r p^{\alpha_r}}  = \frac{(q;q)^n \prod_{j=1}^n (t^j;q) \prod_{0\leq r<s\leq 5: \alpha_r+\alpha_s=0} (t^{n-j} t_rt_s;q)}
{2^n n! (t;q)^n\prod_{j=1}^n  \prod_{0\leq r<s\leq 5:\alpha_r+\alpha_s=1}(q t^{j-n} t_r^{-1}t_s^{-1});q)} \\ &\qquad  \times 
\int_{C^n} \lc(f)(z_i) \lc(g)(z_i) 
\prod_{1\leq j<k\leq n} \frac{(z_j^{\pm 1} z_k^{\pm 1};q)}{(tz_j^{\pm 1}z_k^{\pm 1};q)}
\prod_{j=1}^n \frac{(z_j^{\pm 2};q) \prod_{r:\alpha_r=1} (q t_r^{-1} z_j ^{\pm 1};q)}{\prod_{r:\alpha_r=0} (t_rz_j^{\pm 1};q)} \frac{dz_j}{2\pi i z_j},
\end{align*}
where $\lc(f)=\lc(f(z_i;p^{\alpha_r}t_r))$ and likewise for $\lc(g)$.
Here the integration contour $C=C^{-1}$ is such that it includes the points $q^j \tilde t_r$, (for $0\leq r\leq 5$ with $\alpha_r=0$ and $j\geq 0$)  excludes their reciprocals, and contains $q^jtC$ ($j\geq 0$). The contour can be taken to be the unit circle if $|\tilde t_r|<1$ for all $r$ with $\alpha_r=0$.

\end{proposition}
\begin{proof}
Notice that all the poles of the integrand (of $\langle f,g\rangle_{t_r p^{\alpha_r}}$) which have to be included in the contour are
either $p$-independent or go to 0 as $p\to 0$; while all the poles we have to exclude from the contour are either $p$-independent or go to $\infty$ as $p\to 0$. 
In particular it is possible to find a constant (i.e. independent of $p$) contour (at least for small enough $p$), which works. Then we notice that
the integrand is holomorphic in some neighborhood of the contour 
and converges uniformly to the integrand of the integral on the right hand side of the equation. Therefore we may interchange limit and integral and obtain the desired result.
\end{proof}

The most general case of this kind is when $\alpha=(0,0,0,0,0,1)$ (or a permutation hereof). In this case the limiting measure is given by (using the balancing condition to solve for $t_5$)
\begin{align*}
& \frac{(q;q)^n \prod_{j=1}^n (t^j;q) 
\prod_{0\leq r<s\leq 4} (t^{n-j} t_rt_s;q)}
{2^n n! (t;q)^n\prod_{j=1}^n  \prod_{r=0}^4 (t^{n-2+j} \frac{\prod_{s=0}^4 t_s}{t_r} );q)} \\ &\qquad  \times 
\int_{C^n} \lc(f)(z_i) \lc(g)(z_i) 
\prod_{1\leq j<k\leq n} \frac{(z_j^{\pm 1} z_k^{\pm 1};q)}{(tz_j^{\pm 1}z_k^{\pm 1};q)}
\prod_{j=1}^n \frac{(z_j^{\pm 2};q) (t^{2(n-1)} \prod_{r=0}^4 t_r z_j ^{\pm 1};q)}{\prod_{r=0}^4 (t_rz_j^{\pm 1};q)} \frac{dz_j}{2\pi i z_j}.
\end{align*}
The measure thus corresponds to a multivariate version of the Nasrallah-Rahman integral evaluation \cite[(6.4.1)]{GR} which was first shown by Gustafson \cite[Theorem 2.1]{Gustafson}. Further limits can be obtained by setting the appropriate number of $t_r$ ($0\leq r\leq 4$) equal to 0, which is the practical application of the iterated limit theorem, Proposition \ref{propiteratedlim}.

To get other limits we first observe the identity \cite[Lemma 6.2]{Rainstrafo}, valid for $t^{n-1}v_0v_1v_2v_3=q$:
\[
\sum_{\sigma \in \{\pm 1\}^n}
\prod_{1\leq i<j\leq n} \frac{\theta(tz_i^{\sigma_i} z_j^{\sigma_j};q)}{\theta(z_i^{\sigma_i}z_j^{\sigma_j};q)}
\prod_{i=1}^n \frac{\theta(v_0z_i^{\sigma_i},v_1z_i^{\sigma_i},v_2z_i^{\sigma_i},v_3z_i^{\sigma_i};q)}{\theta(z_i^{2\sigma_i};q)}
= \prod_{i=0}^{n-1} \theta(t^i v_0v_1,t^i v_0v_2,t^i v_1v_2;q).
\]
In particular this allows us to break $z_i\to 1/z_i$ symmetry in the definition of the bilinear form (note that not only the measure of the bilinear
form is $z_i \to 1/z_i$ symmetric, but by assumption the functions $f$ and $g$ which we are allowed to plug in, are as well). 
This gives the identity
\begin{multline}\label{eqlim000001}
\langle f,g\rangle_{t_0,t_1,t_2,t_3,t_4,t_5:q,t;p} = 
\frac{(q;q)^n (p;p)^n \Gamma(t;p,q)^n}{n! \prod_{j=1}^n\Gamma(t^j;p,q) \prod_{0\leq r<s\leq 5} \Gamma(t^{n-j} t_rt_s;p,q)
\theta(t^{n-j} v_0v_1,t^{n-j} v_0v_2,t^{n-j} v_1v_2;q)} \\ \times 
\int_{C^n} f(z) g(z) \prod_{1\leq j<k\leq n} \frac{\Gamma(tz_j^{\pm 1} z_k^{\pm 1};p,q) }{\Gamma(z_j^{\pm 1}z_k^{\pm 1};p,q)}
\frac{\theta(tz_iz_j;q)}{\theta(z_iz_j;q)} 
\prod_{j=1}^n \frac{\prod_{r=0}^5 \Gamma(t_rz_j^{\pm 1};p,q)}{\Gamma(z_j^{\pm 2};p,q)} 
\frac{\theta(v_0z_i,v_1z_i,v_2z_i,v_3z_i;q)}{\theta(z_i^2;q)}
\frac{dz_j}{2\pi i z_j}.
\end{multline}
Here the contours have not changed. Note that we multiply by a function which has (simple) poles only where $\theta(z_jz_k;q)=0$ or
$\theta(z_j^2;q)=0$; which are locations where the original integrand vanished, so the poles of the new integrand are a subset of the poles of the old integrand.
We can simplify this somewhat by using the difference equation for the elliptic gamma function. Different choices of 
$v_r$ moreover allow for different simplifications, thus we prefer to specialize before carrying out the simplifications.

There are several specializations (for the $v_i$) of interest, which can be put in two groups. The first group consists of 
specializing $v_0=\tilde t_r$ and $v_1=\tilde t_s$ for some given $r$ and $s$, while leaving $v_2$ and $v_3$ be free variables 
(though satisfying the balancing condition $v_2v_3 = qt^{1-n} \tilde t_r^{-1}\tilde t_s^{-1}$). The second group consists of specializing 
$v_0=\tilde t_r$, $v_1=\tilde t_s$ and $v_2=\tilde t_w$ for some $r$, $s$ and $w$, which determines $v_3=t^{1-n}/\tilde t_r\tilde t_s\tilde t_w$. 

It should be noted that there is a qualitative difference for whether we specialize the $v$'s as $\tilde t_r$'s with $0\leq r\leq 3$, or
whether we specialize them as $\tilde t_4$ or $\tilde t_5$, which is not immediately clear from the formulas we will give below. 
The difference consists in the interaction with the poles of the function $f$ and $g$, and thus the kind of behavior we allow those
poles to have in the limit. 

The formula we get if we specialize $v_0=\tilde t_a$ and $v_1=\tilde t_b$ (with $a\neq b$) is 
\begin{multline}\label{eqspecialized2}
\langle f,g\rangle_{t_0,t_1,t_2,t_3,t_4,t_5:q,t;p} = 
\frac{(q;q)^n (p;p)^n \Gamma(t;p,q)^n
(-\frac{q}{t_at_b})^{n(m_a+m_b)} t^{-\binom{n}{2}(m_a+m_b)} q^{n\binom{m_a+m_b}{2}}
}{ n! \prod_{j=1}^n\Gamma(t^j;p,q) \Gamma(pt^{n-j}t_at_b) \prod_{\substack{0\leq r<s\leq 5 \\  \{r,s\} \neq \{a,b\}}} \Gamma(t^{n-j} t_rt_s;p,q) \theta(t^{n-j} \tilde t_av_2,t^{n-j} \tilde t_bv_2;q)} \\ \times 
\int_{C^n} \hat f(z) \hat g(z) \prod_{1\leq j<k\leq n} \frac{\Gamma(ptz_jz_k, t z_j/z_k, tz_k/z_j, t/z_jz_k;p,q) }{\Gamma(pz_j z_k,z_j/z_k,z_k/z_j,1/z_jz_k;p,q)}
\\ \times 
\prod_{j=1}^n \frac{\Gamma(p\tilde t_az_j, \tilde t_a/z_j, p\tilde t_bz_j,\tilde t_b/z_j) \prod_{\substack{0\leq r\leq 5 \\ r\neq a,b}} \Gamma(\tilde t_rz_j^{\pm 1};p,q)}{\Gamma(pz_j^2, 1/z_j^2;p,q)} \theta(v_2z_j, v_3z_j;q)
\frac{dz_j}{2\pi i z_j}.
\end{multline}
Specializing $v_0=\tilde t_a$, $v_1=\tilde t_b$ and $v_2=\tilde t_c$ (with $a$, $b$, and $c$ pairwise different) gives the equation
\begin{multline}\label{eqspecialized3}
\langle f,g\rangle_{t_0,t_1,t_2,t_3,t_4,t_5:q,t;p} = 
\frac{(q;q)^n (p;p)^n \Gamma(t;p,q)^n
(\frac{q^2}{t^{n-1}t_at_bt_c} )^{n(m_a+m_b+m_c)}
t_a^{-nm_a} t_b^{-nm_b} t_c^{-nm_c}
}{ n! \prod_{j=1}^n\Gamma(t^j;p,q) 
\Gamma(pt^{n-j} t_at_b,pt^{n-j} t_at_c,pt^{n-j} t_bt_c) \prod_{\substack{0\leq r<s\leq 5 \\ \{r,s\} \not \subset \{a,b,c\} }} \Gamma(t^{n-j} t_rt_s;p,q)} \\ \times  q^{n(\binom{m_a+m_b}{2} + \binom{m_a+m_c}{2}+\binom{m_b+m_c}{2})}
\int_{C^n} \hat f(z) \hat g(z) 
\prod_{1\leq j<k\leq n} \frac{\Gamma(ptz_jz_k, t z_j/z_k, tz_k/z_j, t/z_jz_k;p,q) }{\Gamma(pz_j z_k,z_j/z_k,z_k/z_j,1/z_jz_k;p,q)}
\\ \times \prod_{j=1}^n \frac{\Gamma(p\tilde t_az_j,\tilde t_a/z_j,p\tilde t_bz_j,\tilde t_b/z_j,p\tilde t_cz_j,\tilde t_c/z_j) \prod_{\substack{0\leq r\leq 5  \\ r\neq a,b,c  }} \Gamma(\tilde t_rz_j^{\pm 1};p,q)}{\Gamma(pz_j^{2},1/z_j^2;p,q)} 
\theta(qt^{1-n}z_i/\tilde t_a\tilde t_b\tilde t_c;q)
\frac{dz_j}{2\pi i z_j}.
\end{multline}
Now we can use these two forms of the integrand to identify more cases in which we can write the integral with a fixed contour (as $p$ becomes small), which allows us to interchange limit and integral when taking the limit of the bilinear form.
\begin{proposition}\label{proplimip2}
Let $t_r\in \mathbb{C}$ be generic such that $t^{2(n-1)} \prod_r t_r=q$. 
Let $0\leq a,b\leq 5$ and let $\alpha \in \mathbb{R}^6$, and $-\frac12\leq \zeta<0$ satisfy $\sum_{r=0}^5 \alpha_r=1$, $\alpha_a=\alpha_b =\zeta$ and $-\zeta\leq \alpha_r\leq 1+\zeta$ for $r\neq a,b$. 

Then we have the limit
\begin{align*}
\lim_{p\to 0} & p^{-\val(f)-\val(g)} \langle f(z;t_r p^{\alpha_r})  ,g(z;t_rp^{\alpha_r}) \rangle_{t_r p^{\alpha_r}} \\&   = 
\frac{(q;q)^n \prod_{j=1}^n (t^j;q) (t^{n-j} t_at_b;q)^{1_{\zeta=-1/2}} 
\prod_{r:\alpha_r=-\zeta} (t^{n-j} t_rt_a,t^{n-j}t_rt_b;q)}
{n! (t;q)^n \prod_{j=1}^n  \prod_{\substack{0\leq r<s\leq 5 \\ \alpha_r+\alpha_s=1}} (q t^{j-n} t_r^{-1}t_s^{-1};q) \theta(t^{n-j} t_av,t^{n-j} t_bv;q)} \\ & \qquad \times 
\int_{C^n} \lc(f)(z) \lc(g)(z) 
\prod_{1\leq j<k\leq n} \frac{(z_j/z_k,z_k/z_j;q)}{(tz_j/z_k,tz_k/z_j;q)}
\left( \frac{(z_jz_k,qz_jz_k/t;q)  }{(tz_jz_k,qz_jz_k;q)  }   \right)^{1_{\zeta=-1/2}}
\\ & \qquad \qquad \times \prod_{j=1}^n \frac{\prod_{r:\alpha_r=1+\zeta} (qz_j/t_r;q)}{(t_a/z_j,t_b/z_j;q) \prod_{r: \alpha_r=-\zeta} (t_rz_j;q)}
\left( \frac{(z_j^2;q)}{(t_az_j,t_bz_j, qz_j^2;q)}  \right)^{1_{\zeta=-1/2}} \theta(vz_j, \frac{qz_j}{t^{n-1}vt_at_b};q)
\frac{dz_j}{2\pi i z_j}
\end{align*}
for arbitrary $v\in \mathbb{C}^*$. Here $\lc(f)=\lc(f(zp^{\zeta};t_rp^{\alpha_r}))$ and likewise for $\lc(g)$ and their valuations.

Here the contour $C$ contains the points $q^j \tilde t_a$ and $q^j \tilde t_b$ (for $j\geq 0$), while excluding $q^{-j}/ \tilde t_r$ (for $j\geq 0$ and $r$ such that $\alpha_r=-\zeta$) and, if $\zeta=-1/2$, excluding $q^{-j}/\tilde t_a$ and $q^{-j}/\tilde t_b$ for $j\geq 0$. Moreover $C$ should contain the contours $tq^jC$ (for $j\geq 0$).
\end{proposition}
\begin{proof}
Same as that of the previous proposition, except now we start with \eqref{eqspecialized2} and replace $z_i \to p^{\zeta} z_i$ (and move the contour along). We also specialize $v_2 \to v p^{-\zeta}$, and hence $v_3\to q t^{1-n}\tilde t_a^{-1} \tilde t_b^{-1} v^{-1} p^{-\zeta}$. Only after this substitution the contour can be chosen independently of $p$ for sufficiently small $p$. The limiting integrand now contains $\lc(\hat f)$ and $\lc(\hat g)$, so we express those in terms of $\lc(f)$ and $\lc(g)$ using \eqref{eqlcfinlchatf} and simplify. (This simplification is rather tedious, and you may want to look at the case $a=4,5$ separately from the case $a\neq 4,5$, and likewise for $b$. In the case neither $a=4,5$ nor $b=4,5$ one can avoid the detour along the $\hat f$'s and $\hat g$'s and replace these functions with $f$ and $g$ before taking the limit.)
%
%
\end{proof}
The most general case in this proposition is given for
$\alpha=(-\frac12,-\frac12,\frac12,\frac12,\frac12,\frac12)$ (or a
permutation thereof), in which case we obtain the measure
\begin{align}\label{eqlim--++++}
& \frac{(q;q)^n \prod_{j=1}^n (t^j;q) (t^{n-j} t_0t_1;q)
\prod_{r=2}^5 (t^{n-j} t_rt_0,t^{n-j}t_rt_1;q)}
{n! (t;q)^n \prod_{j=1}^n  \prod_{2\leq r<s\leq 5} (q t^{j-n} t_r^{-1}t_s^{-1};q) \theta(t^{n-j} t_0v,t^{n-j} t_1v;q)} \\ & \qquad \times 
\int_{C^n} \lc(f)(z) \lc(g)(z) 
\prod_{1\leq j<k\leq n} \frac{(z_j/z_k,z_k/z_j, z_jz_k, qz_jz_k/t;q)}{(tz_j/z_k,tz_k/z_j,tz_jz_k,qz_jz_k;q)}   \nonumber
\\ & \qquad \qquad \times \prod_{j=1}^n    \frac{(z_j^2;q) }{( qz_j^2,t_0z_j^{\pm 1},t_1z_j^{\pm 1};q) }
\prod_{r=2}^5 \frac{(qz_j/t_r;q) }{ (t_rz_j;q) }
 \theta(vz_j, \frac{qz_j}{t^{n-1}vt_0t_1};q)
\frac{dz_j}{2\pi i z_j}  \nonumber
\end{align}
To obtain the other cases using iterated limits is slightly non-trivial as it involves rescaling the integration variables $z_i$ along with the parameters. In particular the next most general case is $\alpha=(-\frac14,-\frac14,\frac14,\frac14,\frac14,\frac34)$ and is given by the measure (solving for $t_5$ using the balancing condition)
\begin{align}\label{eqlim--+++1+}
& \frac{(q;q)^n \prod_{j=1}^n (t^j;q)  
\prod_{r=2}^4 (t^{n-j} t_rt_0,t^{n-j}t_rt_1;q)}
{n! (t;q)^n \prod_{j=1}^n  \prod_{r=2}^4  (t^{n-2+j} \frac{\prod_{s=0}^4 t_s}{t_r} ;q) \theta(t^{n-j} t_0v,t^{n-j} t_1v;q)} 
\int_{C^n} \lc(f)(z) \lc(g)(z) 
\\ & \qquad \qquad \times 
\prod_{1\leq j<k\leq n} \frac{(z_j/z_k,z_k/z_j;q)}{(tz_j/z_k,tz_k/z_j;q)}
\prod_{j=1}^n \frac{(t^{2(n-1)} \prod_{r=0}^4 t_r z_j;q)}{(t_0/z_j,t_1/z_j;q) \prod_{r=2}^4 (t_rz_j;q)}
\theta(vz_j, \frac{qz_j}{t^{n-1}vt_0t_1};q)
\frac{dz_j}{2\pi i z_j}. \nonumber 
\end{align}
Further limits of this form can now be obtained by setting $t_r=0$ ($2\leq r\leq 4$). By the iterated limit theorem this expression should also be a limit of \eqref{eqlim000001}, and, while the limit involves shifting the integration variables $z_i$, it is not very hard to obtain that limit directly

We also have the following proposition using \eqref{eqspecialized3}.
\begin{proposition}\label{proplimip3}
Let $t_r\in \mathbb{C}$ be generic, satisfying $t^{2(n-1)} \prod_r t_r=q$. Let $0\leq a,b,c\leq 5$ and let $\alpha \in \mathbb{R}^6$, and $-\frac12\leq \zeta<0$ satisfy $\sum_{r=0}^5 \alpha_r=1$, $\alpha_a+\alpha_b+\alpha_c =\zeta$ and
$\zeta\leq \alpha_r\leq -\zeta$ for $r=a,b,c$ and
$-\zeta\leq \alpha_r\leq 1+\zeta$ for $r\neq a,b,c$.
Then we have

\begin{align*}
\lim_{p\to 0} & p^{-\val(f)-\val(g)}\langle f(z;t_r p^{\alpha_r}) ,g(z;t_rp^{\alpha_r}) \rangle_{t_r p^{\alpha_r}} \\ & = 
\frac{(q;q)^n \prod_{j=1}^n (t^j;q) \prod_{\substack{r,s\in \{a,b,c\}\\ \alpha_r+\alpha_s=-1}} (t^{n-j} t_rt_s;q) 
\prod_{\substack{r \in \{a,b,c\}, s\not \in \{a,b,c\} \\ \alpha_r+\alpha_s=0}} (t^{n-j} t_rt_s;q)
 }{n! (t;q)^n \prod_{j=1}^n 
\prod_{\substack{r,s\in \{a,b,c\}\\ \alpha_r+\alpha_s=0}} (qt^{j-n} t_r^{-1} t_s^{-1};q) 
\prod_{\substack{r,s\\ \alpha_r+\alpha_s=1}} (qt^{j-n} t_r^{-1}t_s^{-1};q)
}
\\& \qquad \times 
\int_{C^n} \lc(f)(z) \lc(g)(z) 
\prod_{1\leq j<k\leq n} \frac{(z_j/z_k,z_k/z_j;q)}{(tz_j/z_k,tz_k/z_j;q)}
\left( \frac{(z_jz_k,qz_jz_k/t;q)  }{(tz_jz_k,qz_jz_k;q)  }   \right)^{1_{\zeta=-1/2}}
\\& \qquad \qquad  \times \prod_{j=1}^n
\frac{\prod_{\substack{r\in \{a,b,c\} \\ \alpha_r+ \zeta=0}} (q/t_rz_j;q) \prod_{\substack{r\not\in \{a,b,c\} \\ \alpha_r=1+\zeta}} (qz_j/t_r;q)}
{\prod_{\substack{r\in \{a,b,c\} \\  \alpha_r=\zeta}} (t_r/z_j;q) \prod_{\substack{r\not \in \{a,b,c\} \\ \alpha_r=-\zeta}} (t_rz_j;q)}
\left( \frac{ (z_j^2;q) \prod_{\substack{r\in \{a,b,c\} \\ \alpha_r=1/2}} (qz_j/t_r;q)}{(qz_j^2;q)\prod_{\substack{r\in \{a,b,c\} \\ \alpha_r=-1/2}} (t_rz_j;q)}  \right)^{1_{\zeta=-1/2}}
\\ & \qquad \qquad \times \theta(qt^{1-n}z_j/t_at_bt_c;q)
\frac{dz_j}{2\pi i z_j},
\end{align*}
where $\lc(f)=\lc(f(p^{\zeta}z))$ and similarly for $g$. We also have the usual conditions on the integration contour.
\end{proposition}
\begin{proof}
The proof is identical to the proof of the previous proposition, but now we start with \eqref{eqspecialized3}.
\end{proof}
The most general form of this proposition is for $\alpha=(-\frac12,-\frac12,\frac12,\frac12,\frac12,\frac12)$, which gives the limit \eqref{eqlim--++++}, 
specialized at $v=t_2$. And we can get the further limits of \eqref{eqlim--+++1+} as well, except that we cannot set $t_2=0$ after this specialization.
The most general limit for which we have to use this proposition is $\alpha=(-\frac12,0,0,\frac12,\frac12,\frac12)$, in which case the bilinear form becomes
\begin{align*}
&\frac{(q;q)^n \prod_{j=1}^n (t^j;q) \prod_{r=3}^5 (t^{n-j} t_0t_r;q) }
{n! (t;q)^n \prod_{j=1}^n  (t^{n-2+j} t_0t_3t_4t_5;q) \prod_{3\leq r<s\leq 5 } (qt^{j-n} t_r^{-1}t_s^{-1};q) }
\int_{C^n} \lc(f)(z) \lc(g)(z) 
\\& \qquad \times 
\prod_{1\leq j<k\leq n} 
\frac{(z_j/z_k,z_k/z_j, z_jz_k, qz_jz_k/t;q)}{(tz_j/z_k,tz_k/z_j,tz_jz_k,qz_jz_k;q)}
\prod_{j=1}^n \frac{(z_j^2;q) \prod_{r=3}^5  (qz_j/t_r;q)}
{(qz_j^2,t_0 z_j^{\pm 1} ;q) \prod_{r=3}^5 (t_rz_j;q)}  \theta(t^{n-1}t_3t_4t_5z_j;q)
\frac{dz_j}{2\pi i z_j}.
\end{align*}
One gets another limiting measure by setting $t_0=0$ in the above measure. The remaining cases can be obtained from the limit for $\alpha=(-\frac13,0,0,\frac13,\frac13,\frac23)$, 
by setting some of $t_0$, $t_3$, and $t_4$ to 0, while keeping $t_3t_4t_5$ constant (in particular once one sets either $t_3$ or $t_4$ equal to 0, one must set $t_5=\infty$). The limit for this $\alpha$ is given by 
\begin{align*}
& \frac{(q;q)^n \prod_{j=1}^n (t^j;q) (t^{n-j} t_0t_3,t^{n-j}t_0t_4;q) }{n! (t;q)^n \prod_{j=1}^n (t^{n-2+j} t_0t_3t_4t_5,qt^{j-n} \frac{t_3}{t_3t_4t_5},qt^{j-n} \frac{t_4}{t_3t_4t_5};q) }
\\& \qquad \times 
\int_{C^n} \lc(f)(z) \lc(g)(z) 
\prod_{1\leq j<k\leq n} \frac{(z_j/z_k,z_k/z_j;q)}{(tz_j/z_k,tz_k/z_j;q)}
\prod_{j=1}^n \frac{ (qz_j/t_5;q)}{ (t_0/z_j,t_3z_j,t_4z_j;q)} \theta(t^{n-1}t_3t_4t_5z_j ;q) \frac{dz_j}{2\pi i z_j}.
\end{align*}

The final case we have to consider is a lot more complicated. In particular we are unable to find an integral expression for
the limiting bilinear form. Indeed, in order to take the limit, we must pick up residues corresponding to poles depending on one
particular variable $t_r$, and take the limit of the residues. The measure for the bilinear form thus turns into a multivariate sum.
In the case the variable we have to take residues of is associated to the pole sequences of one of the two functions, we moreover have to 
take residues of that function in the process.

\begin{proposition}\label{propsumlim}
Let $t_r\in \mathbb{C}$ be generic such that $t^{2(n-1)}t_0t_1t_2t_3t_4t_5=q$.
Let $\sum_{r=0}^5 \alpha_r=1$. Let $0\leq a\leq 5$ be such that $-\frac12\leq \alpha_a<0$ and $1+\alpha_a\geq \alpha_r > \alpha_a$ for $r\neq a$ and such that 
$1\geq \alpha_r+\alpha_s> 0$ for $r,s\neq a$. 
Moreover assume $2\alpha_a = \sum_{r\neq a: \alpha_r+\alpha_a <0} (\alpha_r+\alpha_a)$.

We have if $a\leq 3$ 
\begin{align*}
\lim_{p\to 0} & p^{-\val(f)-\val(g)} \langle f(\cdot;t_rp^{\alpha_r}),g(\cdot;t_rp^{\alpha_r}) \rangle_{t_r p^{\alpha_r}}  =
\prod_{j=0}^{n-1}  \frac{\prod_{r:\alpha_r=1+\alpha_a} (qt^{j} t_a/ t_r;q) } {(q t^{n-1+j} t_a^2;q)^{1_{\{\alpha_a=-1/2\}}} 
\prod_{0\leq r<s\leq 5: \alpha_r+\alpha_s=1}(q / t^{j} t_r t_s;q)} 
\\ & \times 
\sum_{\lambda} \lc(f)(q^{\lambda_i} t^{n-i}  t_a) \lc(g)(q^{\lambda_i} t^{n-i} t_a)
\tilde \Delta_\lambda(t^{2(n-1)} t_a^2)^{1_{\{\alpha_a=-1/2\}}}
\\ &  \times 
\left( \frac{\tilde C_{\lambda}^0(t^n)}{\tilde C_{\lambda}^-(q,t)} 
(-t^{5(n-1)}  t_a^4 q^2)^{-|\lambda|} q^{-3n(\lambda')} t^{5n(\lambda)} \right)^{1_{\{\alpha_a\neq -1/2\}}}
\\ &  \times 
 \frac{ \prod_{r\neq a: \alpha_r=-\alpha_a} \tilde C_{\lambda}^0 (t^{n-1} t_a t_r)}{
 \prod_{r\neq a: \alpha_r=1+\alpha_a} \tilde C_{\lambda}^0(q t^{n-1} t_a/t_r)}
\left((-1)^{A+1} t^{(n-1)(A+3)} t_a^{A+2} q^{2} \prod_{r\neq a: \alpha_r+\alpha_a<0}  t_r  \right)^{|\lambda|} q^{(A+1) n(\lambda')} t^{-(A+1) n(\lambda)}
\end{align*}
where $A=|\{ r~|~ r\neq a,  \alpha_r<-\alpha_a\}|$, 
and $\lc(\hat f)=\lc(\hat f(zp^{\alpha_a}))$, and likewise for $\lc(\hat g)$ and the valuations. The summation is over all partitions $\lambda$.

For general $a$ we have
\begin{align*}
\lim_{p\to 0} & p^{-\val(f)-\val(g)} \langle f(\cdot;t_rp^{\alpha_r}),g(\cdot;t_rp^{\alpha_r}) \rangle_{t_r p^{\alpha_r}}  =
\prod_{j=0}^{n-1}  \frac{\prod_{r:\alpha_r=1+\alpha_a} (qt^{j} \tilde t_a/\tilde t_r;q) } {(q t^{n-1+j} \tilde t_a^2;q)^{1_{\{\alpha_a=-1/2\}}} 
\prod_{0\leq r<s\leq 5: \alpha_r+\alpha_s=1}(q / t^{j} t_r t_s;q)} \\ & \times 
\prod_{j=0}^{n-1} \prod_{r:\alpha_r=-\alpha_a} \frac{(t^j t_rt_a;q)}{(t^j \tilde t_r \tilde t_a;q)}
\prod_{r:\alpha_r<-\alpha_a} \left( - \frac{q}{t^j t_a t_r} \right)^{m_r+m_a} q^{\binom{m_r+m_a}{2}}
\\ & \times 
\sum_{\lambda} \lc(\hat f)(q^{\lambda_i} t^{n-i} \tilde t_a) \lc(\hat g)(q^{\lambda_i} t^{n-i} \tilde t_a)
\\ &  \times 
\tilde \Delta_\lambda(t^{2(n-1)} \tilde t_a^2)^{1_{\{\alpha_a=-1/2\}}}
\left( \frac{\tilde C_{\lambda}^0(t^n)}{\tilde C_{\lambda}^-(q,t)} 
(-t^{5(n-1)} \tilde t_a^4 q^2)^{-|\lambda|} q^{-3n(\lambda')} t^{5n(\lambda)} \right)^{1_{\{\alpha_a\neq -1/2\}}}
\\ &  \times 
 \frac{ \prod_{r\neq a: \alpha_r=-\alpha_a} \tilde C_{\lambda}^0 (t^{n-1}\tilde t_a \tilde t_r)}{
 \prod_{r\neq a: \alpha_r=1+\alpha_a} \tilde C_{\lambda}^0(qt^{n-1}\tilde t_a/\tilde t_r)}
\left((-1)^{A+1} t^{(n-1)(A+3)} \tilde t_a^{A+2} q^{2} \prod_{r\neq a: \alpha_r+\alpha_a<0} \tilde t_r  \right)^{|\lambda|} q^{(A+1) n(\lambda')} t^{-(A+1) n(\lambda)},
\end{align*}
with the same $A$ as before.
\end{proposition}
There exists a similar expression if the $\alpha_r$ are such that we have $\alpha_r+\alpha_s=0$ for some $r,s\neq a$. However in this case the limit is only valid under an extra condition on the functions $f$ and $g$ (to ensure convergence). In particular if we want to plug in our elliptic hypergeometric biorthogonal 
functions $(\langle \tilde R_{\mu},\tilde R_{\nu} \rangle$), this limit would (for some choices of $\alpha$) only hold for small partitions $\mu$ and $\nu$. Fortunately those cases are also treated in Proposition \ref{proplimip3}.

It should also be noted that if we have one parameter $\alpha_r=\alpha_a$ ($r\neq a$) we can get a discrete measure as limit, which is much more complicated in structure than a simple sum over partitions. This would be a multivariate analogue of the double sums appearing in \cite[Proposition 4.3]{vdBR}. An explicit formulation of a measure of this type is given in \cite[Section 7]{Stokman}, which deals with multivariate Big $q$-Jacobi polynomials (corresponding to the vector $\alpha=(-\frac16,-\frac16,\frac16,\frac16;\frac12,\frac12)$). Fortunately we also have an integral limit for these cases given in Proposition \ref{proplimip2}. Shrinking the contour of the integral in the latter proposition gives an expression of the measure as a sum of residues, which is equal to this multivariate analogue of a double sum measure. While the integral expression is simpler, it should be noted that, unlike the discrete measure for multivariate Big $q$-Jacobi polynomials it can not be made positive. 

The most general case of this proposition is $\alpha=(-\frac12,\frac16,\frac16,\frac16,\frac12,\frac12)$, in which case the measure becomes (since we assume $a=0$ we can use the first expression, otherwise the measure looks very similar but slightly more complicated). 
\begin{align*}
& \prod_{j=0}^{n-1}  \frac{(qt^{j} t_0/ t_4,qt^{j} t_0/ t_5;q) } 
{(q t^{n-1+j} t_0^2,q / t^{j} t_4 t_5;q)} 
\sum_{\lambda} \lc(f)(q^{\lambda_i} t^{n-i}  t_0) \lc(g)(q^{\lambda_i} t^{n-i} t_0)
\\ & \times 
 \frac{\tilde C_{2\lambda^2}^0(qt^{2(n-1)}t_0^2;q,t) \tilde C_{\lambda}^0(t^n,t^{n-1} t_0 t_4,t^{n-1} t_0 t_5) }{
 \tilde C_{\lambda}^-(q,t;q,t) \tilde C_{\lambda}^+(t^{2(n-1)}t_0^2,qt^{2n-3}t_0^2;q,t)  \tilde C_{\lambda}^0(qt^{2n-2}t_0^2,q t^{n-1} \frac{t_0}{t_4},q t^{n-1} \frac{t_0}{t_5})}
\left(-\frac{q}{t^{n-1}t_4t_5} \right)^{|\lambda|} q^{n(\lambda')} t^{n(\lambda)}
\end{align*}
In this measure we can take further limits by letting $t_4$ and/or $t_5$ tend to infinity. The final measure we want to show is the most general measure
with $\alpha_a>-\frac12$, associated to $\alpha=(-\frac38,\frac18,\frac18,\frac18,\frac38,\frac58)$ and is given by 
\begin{align*}
& \prod_{j=0}^{n-1}  \frac{(qt^{j} t_0/ t_5;q) } {(q / t^{j} t_4 t_5;q)} 
\sum_{\lambda} \lc(f)(q^{\lambda_i} t^{n-i}  t_0) \lc(g)(q^{\lambda_i} t^{n-i} t_0)
\frac{\tilde C_{\lambda}^0(t^n,t^{n-1} t_0 t_4)}{\tilde C_{\lambda}^-(q,t)\tilde C_{\lambda}^0(q t^{n-1} t_0/t_5)} 
\left(-\frac{q}{t^{n-1} t_4 t_5}   \right)^{|\lambda|} q^{n(\lambda')} t^{ n(\lambda)}.
\end{align*}
Further limits can be obtained by letting $t_4$ tend to zero or infinity, while keeping $t_4t_5$ constant.

\begin{proof}
Let us start with the expression \eqref{eqacip2} (in terms of $\hat f$ and $\hat g$) for the bilinear form 
without yet replacing the $t_r$'s by $t_r p^{\alpha_r}$, which simplifies the formulas to come. This expression of the bilinear form ensures we don't have to worry about poles of the function $f$ and $g$. We want to do residue calculus on this integral, so we need the following lemma, which is very similar to \cite[Lemma 10.5]{Rainstrafo}, and its proof is identical (see also the discussion leading up to \cite[Theorem 10.7]{Rainstrafo}). 
\begin{lemma}
Let $\Delta$ be a $BC_n$-symmetric meromorphic function on $(\mathbb{C}^*)^n \times P$, where $P$ is an irreducible normal subvariety of the domain $\{a_0,a_1,\ldots,a_{d-1},b_0,\ldots, b_{d-1},p,q,t\in \mathbb{C}^*~|~ |p|,|q|, |t|<1\}$. Suppose furthermore that the following conditions are satisfied
\begin{itemize}
\item The function 
\[
\prod_{i=1}^n \prod_{r=0}^{d-1} (a_r z_i, b_r/z_i;p,q) \prod_{1\leq i<j\leq n} (t z_i^{\pm 1} z_j^{\pm 1};p,q) 
\Delta(z;\textbf{p})
\]
is holomorphic
\item At a generic point of $P$, the factor $\prod_{i=1}^n \prod_{r=0}^{d-1} (a_r z_i, b_r/z_i;p,q) \prod_{1\leq i<j\leq n} (t z_i^{\pm 1} z_j^{\pm 1};p,q)$ has only simple zeros.
\item For any integers $i,j,k,l \geq 0$,
\[
\Delta(p^i q^j z,p^k q^l z,z_3,z_4,\ldots,z_n;\textbf{p}) = 
-\Delta(p^i q^l z,p^k q^j z,z_3,z_4,\ldots,z_n;\textbf{p}) 
\]
as an identity of meromorphic functions on $(\mathbb{C}^*)^{n-1} \times P$. 
\end{itemize}
For generic $\textbf{p} \in P$, choose a contour $C_{\textbf{p}}$ containing all points of the form $b_rp^i q^j$ ($i,j \in \mathbb{Z}_{\geq 0}$, $0\leq r<d$), and excluding all points of the form $(a_rp^i q^j)^{-1}$ ($i,j \in \mathbb{Z}_{\geq 0}$, $0\leq r<d$) and including the contour $tp^i q^j C_{\textbf{p}}$ ($i,j \in \mathbb{Z}_{\geq 0}$). Let $C_{\textbf{p}}'$ be a different contour satisfying the same conditions, except now excluding the points $b_0 p^i q^j$ ($0\leq i\leq l$, $0\leq j\leq m$). Then 
\begin{align*}
\int_{C_{\textbf{p}}^n} \Delta(z;\textbf{p}) \prod_{j=1}^n \frac{dz_j}{2\pi i z_j} &-
\int_{C_{\textbf{p}}'^n} \Delta(z;\textbf{p}) \prod_{j=1}^n \frac{dz_j}{2\pi i z_j} 
\\ &
= n \sum_{r=0}^l \sum_{s=0}^m \int_{C_{\textbf{p}}'^{n-1}}
\lim_{z_n \to p^r q^s b_0} (1- p^r q^s b_0/z_n) \Delta(z;\textbf{p}) \prod_{j=1}^{n-1} \frac{dz_j}{2\pi i z_j} .
\end{align*}
A similar equation holds for moving the contour through the points $(a_0p^i q^j)^{-1}$ ($0\leq i\leq l$, $0\leq j\leq m$).
\end{lemma}
Fix some constant $M$. We can now shift the contour over the poles at $q^{k_1} \tilde t_a$, and simultaneously passing over the pole at $q^{-k_1}/\tilde t_a$ for $0\leq k_1\leq M$. The residue corresponding to $q^{k_1} \tilde t_a$ equals the residue corresponding to $q^{-k_1}/\tilde t_a$ (with a minus sign), so that in fact we get an extra factor 2 by moving over both sets of poles. Thus we obtain  
\begin{align*}
& \int_{C^n}  \hat f(z_i) \hat g(z_i) \prod_{1\leq j<k\leq n} \frac{\Gamma(tz_j^{\pm 1} z_k^{\pm 1})}{\Gamma(z_j^{\pm 1}z_k^{\pm 1})}
\prod_{j=1}^n \frac{\prod_{s=0}^5 \Gamma(\tilde t_sz_j^{\pm 1})}{\Gamma(z_j^{\pm 2})} \frac{dz_j}{2\pi i z_j}
\\ &
= 
\int_{C'^n}  \hat f(z_i) \hat g(z_i) \prod_{1\leq j<k\leq n} \frac{\Gamma(tz_j^{\pm 1} z_k^{\pm 1})}{\Gamma(z_j^{\pm 1}z_k^{\pm 1})}
\prod_{j=1}^n \frac{\prod_{s=0}^5 \Gamma(\tilde t_sz_j^{\pm 1})}{\Gamma(z_j^{\pm 2})} \frac{dz_j}{2\pi i z_j}
\\ & \quad + 
2n \sum_{k_1=0}^{M} \frac{\Gamma(q^{k_1} \tilde t_a^2) \prod_{s\neq a} \Gamma(\tilde t_sq^{k_1}\tilde t_a, q^{-k_1} \frac{\tilde t_s}{\tilde t_a})}{\Gamma(q^{2k_1} \tilde t_a^2, q^{-2k_1} \tilde t_a^{-2})} 
\frac{1}{(p;p)(q;q) \theta(q^{-k_1};p)_{k_1}}
 \\ & \quad \times 
\int_{C'^{n-1}}  \hat f(z_i,q^{k_1} \tilde t_a) \hat g(z_i, q^{k_1} \tilde t_a) 
\prod_{1\leq j<k\leq n-1} \frac{\Gamma(tz_j^{\pm 1} z_k^{\pm 1})}{\Gamma(z_j^{\pm 1}z_k^{\pm 1})}
\prod_{j=1}^{n-1} \frac{\Gamma(tq^{k_1} \tilde t_a z_j^{\pm 1}, \frac{t}{\tilde t_a} q^{-k_1}  z_j^{\pm 1})}{\Gamma(q^{k_1} \tilde t_a z_j^{\pm 1},\frac{1}{\tilde t_a} q^{-k_1} z_j^{\pm 1})}
 \frac{\prod_{s=0}^5 \Gamma(\tilde t_sz_j^{\pm 1})}{\Gamma(z_j^{\pm 2})} 
\frac{dz_j}{2\pi i z_j}
\end{align*}
Let us now observe that
\[
\frac{\Gamma(\tilde t_a z^{\pm 1})}{\Gamma(q^{k_1} \tilde t_a z^{\pm 1}, \frac{1}{\tilde t_a} q^{-k_1} z^{\pm 1})}
= 
\Gamma(\tilde t_a z^{\pm 1}) \theta(\tilde t_a q^{k_1} z^{\pm 1};q) \theta(p \tilde t_a q^{k_1} z^{\pm 1};p)=
\Gamma(p\tilde t_a z^{\pm 1}) \tilde t_a^{-2k_1} q^{-2\binom{k_1}{2}} \theta(p\tilde t_a q^{k_1} z^{\pm 1};p).
\]
Thus we can view each of the residues as an integral of the form of the lemma, and apply the lemma again. Now we pass over the poles at $t \tilde t_a q^{k_1+k_2}$ ($0\leq k_2\leq M-k_1$), and obtain residues which are integrals of dimension $n-2$. 
We can subsequently iterate this until the residues are 0-dimensional integrals (i.e. constants).
We end up with an $n$-fold sum of residues times $0$-dimensional integrals, plus an $(n-1)$-fold sum of univariate integrals, an $(n-2)$-fold sum of bivariate integrals etc. We will later show that every term with an integral vanishes, so we focus now on the $n$-fold sum. 

This sum, where we include the prefactor, is given by (for moving $z_i$ over poles at $q^{\lambda_{j}} t^{n-j} \hat t_a$ for $1\leq j\leq n$) 
\begin{align*}
& \frac{\Gamma(t)^n}{\prod_{j=1}^n \Gamma(t^j) \prod_{0\leq r<s\leq 5} \Gamma(t^{n-j}t_rt_s)}
\sum_{\lambda\subset M^n} \hat f(q^{\lambda_i} t^{n-i} \tilde t_a) \hat g(q^{\lambda_i}t^{n-i} \tilde t_a) \\ & \qquad \times 
\prod_{1\leq i<j\leq n} \frac{\Gamma(\tilde t_a^2 q^{\lambda_i+\lambda_j} t^{2n+1-i-j}, \frac{1}{\tilde t_a^2} q^{-\lambda_i-\lambda_j} t^{1+i+j-2n}, q^{\lambda_i-\lambda_j} t^{1+j-i})}
{\Gamma(\tilde t_a^2 q^{\lambda_i+\lambda_j} t^{2n-i-j}, \frac{1}{\tilde t_a^2} q^{-\lambda_i-\lambda_j} t^{i+j-2n}, q^{\lambda_i-\lambda_j} t^{j-i}, q^{\lambda_j-\lambda_i} t^{i-j})}
\\ & \qquad \times \prod_{i=1}^n \frac{ \Gamma(\tilde t_a^2 q^{\lambda_i} t^{n-i})  \prod_{r\neq a} \Gamma( \tilde t_r \tilde t_a q^{\lambda_i} t^{n-i},\frac{\tilde t_r}{\tilde t_a} q^{-\lambda_i} t^{i-n})}{\Gamma(\tilde t_a^2 q^{2\lambda_i} t^{2(n-i)}, \frac{1}{\tilde t_a^2} q^{-2\lambda_i} t^{-2(n-i)}) }
\prod_{i=1}^{n-1} \Gamma(q^{-\lambda_i} t^{-(n-i)})
\prod_{1\leq i<j-1 \leq n-1} \Gamma(q^{\lambda_j-\lambda_i}  t^{1+i-j})
\\ & \qquad \times 
\frac{1}{\prod_{j=1}^n \theta(q^{-\lambda_{j}+\lambda_{j+1}};p)_{\lambda_{j}-\lambda_{j+1}}}
\end{align*}
where $\lambda_{n+1}=0$ by definition. We can simplify this, by first taking out the $\lambda=0^n$ term to 
\begin{align*}
 \prod_{j=0}^{n-1} & 
\frac{ \Gamma(pq t^{2(n-1)-j} \tilde t_a^2 )}{ \prod_{0\leq r<s\leq 5} \Gamma(t^{j}t_rt_s)}
 \prod_{r\neq a} \frac{\Gamma(t^{j} \tilde t_a\tilde t_r)}{\Gamma(p q t^{j}\frac{\tilde t_a}{\tilde t_r} )}
\\ & \times\sum_{\lambda}  
 \hat f(q^{\lambda_j} t^{n-j} \tilde t_a) \hat g(q^{\lambda_j} t^{n-j} \tilde t_a) 
 \Delta_{\lambda}( t^{2(n-1)} \tilde t_a^2 ~|~ t^n, t^{n-1} \tilde t_a \tilde t_0,\ldots,\widehat{t^{n-1} \tilde t_a \tilde t_a} ,\ldots, t^{n-1} \tilde t_a \tilde t_5)
\end{align*}
where in the arguments of $\Delta_{\lambda}$ we omit the $t^{n-1} \tilde t_a \tilde t_a$ term. This simplification is quite tedious, so will not include it here, except to mention that the alternative expressions for the $C_{\lambda}^*$ from \eqref{eqaltdefcl} come in useful. 
Note that this sum is the same as the measure for the finitely supported measure, apart from a different scaling factor. 

Now we are ready to replace the $t_r$ by $t_rp^{\alpha_r}$, and moreover let $M$ depend on $p$. Choose an $\epsilon>0$ such that $\alpha_a+\epsilon<\alpha_r$ for $r\neq a$, then we set $M= \epsilon \log_{|q|}(|p|)$, so it increases slowly as $p \to 0$ to ensure that $|q^M| = |p^{\epsilon}|$. We will show in the appendix that for the values of $\alpha$ given the terms with integral vanish (even when we consider that their number increases as $p\to 0$), while the limit of the sum term is just obtained by taking the sum of the termwise limits. This gives us the desired result.
\end{proof}

We can summarize the results in this subsection by the following theorem. The polytope $P^{(0)}$ is the same as in \cite{vdBR}. Notice that the interior of $P_{II,0}$ as given below, and its various permutations, are precisely the subpolytopes on which the limits of the biorthogonal functions as discussed in \cite{vdBRmult}, fail to exist. Thus we obtain explicit bilinear forms for almost all of the limits of the biorthogonal functions discussed in \cite{vdBRmult}, and likewise have an explicit family of biorthogonal functions for each of the bilinear forms given in this section. The biorthogonal functions which still lack an explicit bilinear form are the flipped versions of $0022pp$, $04as$ and $0031as$, that is they correspond to Stieltjes-Wiegert and Continuous $q$-Hermite for $|q|<1$ (also known as Continuous $q^{-1}$-Hermite). Explicit bilinear forms for these functions are the bilateral series given in the next section.

It should be stressed that the measures that we find are not necessarily positive, even in cases in which one would expect to have a positive measure (such as real-valued orthogonal polynomials with positive squared norms). Indeed, in a generic case, we expect there are several measures which can be obtained as limits from the elliptic hypergeometric level, other than just the ones we gave here (for example, derived by using techniques as in the next section). We hope that amongst those different options we can find a positive measure in each case where this is expected, but deriving that is reserved for later work.
\begin{theorem}\label{thmacmeasurelim}
Let $t_r\in \mathbb{C}$ be generic such that $\prod_r t_r=q$. 
Consider the polytope $P^{(0)}$ given by the bounding inequalities
\[
\alpha_r\geq -\frac12, \qquad 
\alpha_r-\alpha_s\leq 1, \qquad 
\alpha_r+\alpha_s\leq 1, \qquad 
\sum_r \alpha_r =1.
\]

For each vector in $P^{(0)}$, outside the interior of the subpolytope $P_{II,0}$
\[
-\frac12\leq \alpha_0\leq 0, \qquad 
\alpha_0\leq \alpha_r\leq 1+\alpha_0, \quad (1\leq r\leq 5), \qquad 
0\leq \alpha_r+\alpha_s\leq 1, \quad (1\leq r<s\leq 5), \qquad
\sum_r \alpha_r=1,
\]
or one of its 5 images under permutation of the $\alpha_r$'s, we find
\[
\lim_{p\to 0} \langle f,g\rangle_{t_rp^{\alpha_r}} = 
\langle \lc(f(zp^{\zeta})),\lc(g(zp^{\zeta})) \rangle_{\alpha,t_r}
\]
with $\zeta$ given by $|\zeta+\frac12|=\min(\frac12,\alpha_r+\frac12,\alpha_r+\alpha_s+\alpha_t+\frac12)$ and the limiting inner products $\langle \cdot,\cdot \rangle_{\alpha,t_r}$ being given in the propositions of this section.
\end{theorem}

\section{Bilateral series}\label{sec5}
The final limits we consider are of a slightly different form. In particular these limits only work when we let $p\to 0$ along a geometric progression, and as long as $\alpha_r\in \mathbb{Q}$ for all $r$. Let $d$ be the least even common multiple of the denominators of the $\alpha_r$. We set $p=(xq^k)^d$ for the purpose of this section. We will consider limits $k\to \infty$ (for integer $k$), which corresponds to letting $p\to 0$. We can choose any $q^d$-geometric progression by varying $x$, and we will typically obtain different results for different values of $x$. 

The reason for these conditions is that it allows us to write (for $\alpha \in \frac{1}{d} \mathbb{Z}$)
\[
\theta(p^{\alpha} y;q) = \theta(x^{d\alpha} q^{dk\alpha} y;q) = 
\theta(x^{d\alpha}y;q) (-x^{d\alpha} y)^{-dk\alpha} q^{-\binom{dk\alpha}{2}},
\]
which allows us to determine the behavior as $k\to \infty$, while the behavior as $p\to 0$ continuously is erratic, as we would pass many zeros (whenever $p^{\alpha} \in \frac{1}{y} q^{\mathbb{Z}}$), and many large values (the $q^{-\binom{km\alpha}{2}}$ blows up if $k\to \infty$ for $\alpha\neq 0$). In particular this also gives the limiting behavior of 
$(p^{\alpha} y;q)$ for all $\alpha \in \frac{1}{d} \mathbb{Z}$. If $\alpha>0$ the limits is 1, and for $\alpha<0$  we obtain the limit by writing $(p^{\alpha} y;q) = \theta(p^{\alpha} y;q)/(\frac{q}{y} p^{-\alpha};q)$.

In this section we need the extensions of the definition of $C_{\lambda}^\epsilon$ for $\lambda \in \mathbb{Z}^n$ a decreasing, not necessarily positive sequence. 
If we want to use these functions, we run into the technical difficulty that the expressions we get have spurious poles and zeros which we must cancel to each other. We can solve this issue by defining
\[
D_{\lambda}:=\frac{C_{\lambda}^0(t^{n-1}x)}{C_{\lambda}^-(x)}=\prod_{1\leq i<j\leq n} \frac{\theta(t^{j-i}x;p)_{\lambda_i-\lambda_j}}{\theta(t^{j-1-i}x;p)_{\lambda_i-\lambda_j}}, \qquad 
\tilde D_{\lambda}:=\frac{\tilde C_{\lambda}^0(t^{n-1}x)}{\tilde C_{\lambda}^-(x)}=\prod_{1\leq i<j\leq n} \frac{(t^{j-i}x;q)_{\lambda_i-\lambda_j}}{(t^{j-1-i}x;q)_{\lambda_i-\lambda_j}},
\]
and use $D_{\lambda}$ instead of $C_{\lambda}^-$, and we will do so for the rest of this section. It should be noted that $D_{\lambda}(q)$ is well-defined for all decreasing integer sequences $\lambda$, even though $C_{\lambda}^0(t^{n-1}q)=C_{\lambda}^-(q)=0$ when $\lambda_n<0$.


There are two cases we want to consider. The first one is the following.
\begin{proposition}
Let $t_r\in \mathbb{C}$ be generic such that $t^{2(n-1)}t_0t_1t_2t_3t_4t_5=q$.
Let $\sum_{r=0}^5 \alpha_r=1$. Let $0\leq a,b,c\leq 5$ be such that 
$\alpha_a+\alpha_b>0$, $\alpha_a+\alpha_c>0$, $\alpha_b+\alpha_c>1$, 
$\alpha_a-\beta< \alpha_r\leq \beta-\alpha_a$ ($r\neq a,b,c$),
where $\beta=1-\alpha_b-\alpha_c$.
We have 
\begin{align*}
\lim_{k\to \infty} &  p^{-\val(f)-\val(g)} \langle f(\cdot;t_rp^{\alpha_r}),g(\cdot;t_rp^{\alpha_r}) \rangle_{t_r p^{\alpha_r}} 
\\& =
\prod_{i=1}^n \frac{\prod_{r\neq a,b,c:\alpha_a+\alpha_r=\beta} (\frac{q}{t^{n-i} t_at_r x^{d\beta}};q)
\prod_{r,s\neq a,b,c: \alpha_r+\alpha_s=0} (t^{i-1}t_rt_s;q)}
{\theta( t^{i-1} t_bt_c x^{-d\beta};q) (qt^{n-i};q) }
\\& \qquad \times
\sum_{\lambda\in \mathbb{Z}^n} \lc(f)(q^{\lambda_j} t^{n-j} t_a x^{\beta d})\lc(g)(q^{\lambda_j} t^{n-j} t_a x^{\beta d})
 q^{n(\lambda')} t^{n(\lambda)} (- \frac{q x^{d\beta}}{t^{n-1} t_bt_c} )^{|\lambda|} 
\\& \qquad \times 
\tilde D_{\lambda}(q,t) 
\prod_{r\neq a,b,c: \alpha_a+\alpha_r=\beta} \tilde C_{\lambda}^0(t^{n-1} t_at_r x^{d\beta}) 
\left(-\frac{1}{t^{n-1}t_at_r x^{d\beta}}\right)^{|\lambda|} q^{-n(\lambda')} t^{n(\lambda)}
\end{align*}
where $\lc( f)=\lc(f(zp^{\alpha_a-\beta}))$, and likewise for $\lc(g)$ and the valuations. The summation is over all (weakly) ordered sequences $\lambda \in \mathbb{Z}^n$.
\end{proposition}
Observe that the right hand side depends on $t_a$ and $x$ only through the combination $t_a x^{\beta d}$. As $x$ can be chosen by taking the limit over the appropriate geometric sequence of values for $p$, this means that the right hand side is essentially independent of $t_a$. 

Apart from the choice of $a$, $b$ and $c$ , there are basically two distinct cases of this proposition, associated to $\alpha=(-\frac25,0,0,\frac15,\frac35,\frac35)$, giving
\begin{align*}
& \prod_{i=1}^n \frac{(\frac{q}{t^{n-i} t_0t_3 x^{d\beta}}, t^{i-1}t_1t_2;q)}
{\theta( t^{i-1} t_4t_5 x^{-d\beta};q) (qt^{n-i};q) }
\\& \qquad \times
\sum_{\lambda\in \mathbb{Z}^n} \lc(f)(q^{\lambda_j} t^{n-j} t_0 x^{\beta d})\lc(g)(q^{\lambda_j} t^{n-j} t_0 x^{\beta d})
D_{\lambda}(q,t)   \tilde C_{\lambda}^0(t^{n-1} t_0t_3 x^{d\beta}) 
t^{2n(\lambda)} ( t_1t_2)^{|\lambda|} 
\end{align*}
and to $\alpha=(-\frac{5}{12}, \frac{1}{12},\frac{1}{12},\frac{1}{12}, \frac{7}{12},\frac{7}{12})$ (obtained by letting $t_3\to \infty$ and $t_1t_2\to 0$ while keeping their product constant) giving
\begin{align*}
&\prod_{i=1}^n \frac{ 1}
{\theta( t^{i-1} t_4t_5 x^{-d\beta};q) (qt^{n-i};q) }
\sum_{\lambda\in \mathbb{Z}^n} \lc(f)(q^{\lambda_j} t^{n-j} t_0 x^{\beta d})\lc(g)(q^{\lambda_j} t^{n-j} t_0 x^{\beta d})
D_{\lambda}(q,t)  (- \frac{q x^{d\beta}}{t^{n-1} t_4t_5} )^{|\lambda|}.
\end{align*}
\begin{proof}
For notational convenience we assume $a\neq 4,5$, the proof in these cases still works with slight modifications to accommodate the extra poles introduced by $f$ or $g$ (the corresponding residues vanish in the limit). In particular this allows us to work with the $t_r$ parameters instead of $\tilde t_r$.  As in the proof of Proposition \ref{propsumlim} we first pick up the residues of the poles associated to $t_a$, now for 
$M=dk(-\beta+\epsilon)$ for some $0<\epsilon < -\beta$ to obtain that the the left hand side equals a series of residues plus a sum of residues times integrals. As before the estimates from Lemma \ref{lempiv} together with Lemma \ref{lempiv2} show that the terms involving integrals all vanish in the limit. Thus we are concerned just with the sum of the residues. Now we can rewrite these by shifting the index $\lambda$ as
\begin{align*}
 \prod_{j=0}^{n-1} & 
\frac{ \Gamma(pq t^{2(n-1)-j} t_a^2 )}{ \prod_{0\leq r<s\leq 5} \Gamma(t^{j}t_rt_s)}
 \prod_{r\neq a} \frac{\Gamma(t^{j}  t_a t_r)}{\Gamma(p q t^{j}\frac{t_a}{t_r} )}
\\ & \times\sum_{0^n \subset \lambda\subset M^n}  
f(q^{\lambda_j} t^{n-j} t_a)  g(q^{\lambda_j} t^{n-j} t_a) 
 \Delta_{\lambda}( t^{2(n-1)} t_a^2 ~|~ t^n, t^{n-1} t_a t_0,\ldots,\widehat{t^{n-1}  t_a t_a} ,\ldots, t^{n-1}  t_a  t_5)
 \\ &= 
 \prod_{j=0}^{n-1}  
\frac{ \Gamma(pq t^{2(n-1)-j} t_a^2 )}{ \prod_{0\leq r<s\leq 5} \Gamma(t^{j}t_rt_s)}
 \prod_{r\neq a} \frac{\Gamma(t^{j}  t_a t_r)}{\Gamma(p q t^{j}\frac{t_a}{t_r} )}
\sum_{(\beta dk)^n \subset \lambda\subset (M+\beta dk)^n}  
 f(q^{\lambda_j - \beta dk} t^{n-j} t_a)  g(q^{\lambda_j-\beta dk} t^{n-j} t_a) 
\\ & \qquad \times 
 \Delta_{\lambda- (\beta dk)^n}( t^{2(n-1)} t_a^2 ~|~ t^n, t^{n-1} t_a t_0,\ldots,\widehat{t^{n-1}  t_a t_a} ,\ldots, t^{n-1}  t_a  t_5)
\end{align*}
In this expression we can replace the $t_r\to t_rp^{\alpha_r}$ and rewrite the $C_{\lambda- (\beta dk)^n}^{\epsilon}$, respectively $D_{\lambda - (\beta dk)^n}$, as $C_{\lambda}^\epsilon$, respectively $D_{\lambda}$, times some elliptic gamma functions using the equations 
\begin{align*}
C_{\lambda+m^n}^0(x) &=  C_{\lambda}^0(q^m x)
\prod_{i=1}^n \frac{\Gamma(q^m t^{1-i}x)}{\Gamma(t^{1-i}x)}, & \tilde C_{\lambda+m^n}^0(x) &= 
\tilde C_{\lambda}^0(q^m x) \prod_{i=1}^n \frac{(t^{1-i}x;q)}{(q^m t^{1-i}x;q)}
\\
D_{\lambda+m^n}(x) &= D_{\lambda}(x), & 
\tilde D_{\lambda+m^n}(x) &= \tilde D_{\lambda}(x),
 \\
C_{\lambda+m^n}^+(x) &= C_{\lambda}^+(q^{2m}x) \prod_{i=1}^n \frac{\Gamma(q^{\lambda_i+2m}t^{2-n-i}x)}{\Gamma(q^{\lambda_i+m} t^{2-n-i}x)}, & 
\tilde C_{\lambda+m^n}^+(x) &= \tilde C_{\lambda}^+(q^{2m}x) 
\prod_{i=1}^n \frac{(q^{\lambda_i+m} t^{2-n-i}x;q)}{(q^{\lambda_i+2m}t^{2-n-i}x;q)}.
\end{align*}
Subsequently we replace the $q^{\beta dk}$ appearing by $p^{\beta}/x^{d\beta}$, 
use the difference equation $\Gamma(px)=\theta(x;q)\Gamma(x)$ of the elliptic gamma functions to ensure all elliptic gamma functions are of the form $\Gamma(p^{\alpha} x)$ for $0\leq \alpha\leq 1$ for any choice of $\lambda$, and likewise for the $C_{\lambda}^{\epsilon}$ terms. Then we replace $p\to (xq^k)^d$ in the theta functions which have thus appeared, and use the difference equation of the theta function to make these constant times a certain factor. Thus we obtain the expression
\begin{align*}
& \sum_{(\beta k d)^n \subset \lambda \subset (M+\beta k d)^n}
 f(q^{\lambda_j} x^{\beta d} t^{n-j} t_a p^{\alpha_a-\beta})  g(q^{\lambda_j} x^{\beta d} t^{n-j} t_a p^{\alpha_a-\beta}) 
q^{n(\lambda')} t^{n(\lambda)} (-\frac{q}{t^{n-1}t_bt_c}  x^{d\beta})^{|\lambda|}
\\ & \times 
\frac{D_{\lambda}(q,t)C_{2\lambda^2}^0(p^{1-2\beta+2\alpha_a} q t^{2(n-1)} t_a^2 x^{2d \beta}) C_{\lambda}^0(p^{-\beta} t^n x^{\beta d}) \prod_{r\neq a,b,c}  C_{\lambda}^0(p^{1-\beta+\alpha_a+\alpha_r} t^{n-1} t_at_r x^{\beta d})   }
{ C_{\lambda}^+(p^{1-2\beta+2\alpha_a} q t^{2n-3} t_a^2 x^{2d\beta}, p^{1-2\beta+2\alpha_a} t^{2(n-1)} t_a^2 x^{2d\beta})
C_{\lambda}^0(p^{1-\beta+2\alpha_a} q t^{n-2} t_a^2 x^{d\beta}) 
}
\\ & \times 
\frac{\prod_{r=b,c}
C_\lambda^0(p^{-\beta+\alpha_a+\alpha_r} t_rt_ax^{d\beta}) }{\prod_{r\neq a} C_{\lambda}^0(p^{1-\beta+\alpha_a-\alpha_r} q t^{n-1} \frac{t_a}{t_r} x^{d\beta}) }
\prod_{i=1}^n \frac{\Gamma(q t^{n-i},  p^{-\beta} t^{1+n-i} x^{d\beta}, p^{1-\beta+2\alpha_a} q^{\lambda_i+1} t^{n-i-1} t_a^2 x^{d\beta} )  }
{\Gamma(p^{-\beta} q^{\lambda_i+1} t^{n-i} x^{d\beta}, p^{-\beta} q^{\lambda_i} t^{n+1-i} x^{d\beta}) }
\\& \times \prod_{i=1}^n \frac{\Gamma( p^{1-\beta+2\alpha_0} q^{\lambda_i} t^{n-i} t_a^2 x^{d\beta},p^{1-2\beta+2\alpha_a}q t^{n-i-1} t_a^2 x^{2 d \beta}, p^{1-2\beta+2\alpha_a}qt^{2n-i-1} t_a^2 x^{2d\beta})}{\Gamma( p^{1-\beta+2\alpha_a} q t^{n-i-1} t_a^2 x^{d\beta}, p^{1-2\beta+2\alpha_0} q^{\lambda_i+1} t^{n-i-1} t_a^2 x^{2d\beta},p^{1-2\beta+2\alpha_a} q^{\lambda_i} t^{n-i} t_a^2 x^{2d\beta})}
\\& \times \prod_{i=1}^n \prod_{r\neq a,b,c} \frac{\Gamma( p^{1-\beta+\alpha_a+\alpha_r} t^{n-i} t_at_r x^{d\beta}  )}
{\Gamma( p^{1+\alpha_a+\alpha_r} t^{n-i} t_at_r, p^{1-\beta+\alpha_a-\alpha_r} q t^{n-i} x^{d\beta} \frac{t_a}{t_r})}
\prod_{r=b,c} \frac{\Gamma( p^{-\beta+\alpha_a+\alpha_r} t^{n-i} t_at_r x^{d\beta} )}
{\Gamma( p^{\alpha_a+\alpha_r} t^{n-i} t_at_r, p^{1-\beta+\alpha_a-\alpha_r} q t^{n-i} x^{d\beta} \frac{t_a}{t_r})}
\\ & \times \prod_{i=1}^n   \frac{1}{\Gamma( p^{\alpha_b+\alpha_c-1} t^{i-1}t_bt_c)} \prod_{\substack{0\leq r<s\leq 5 \\ r,s\neq a,b,c}} \frac{1}{\Gamma( p^{\alpha_r+\alpha_s} t^{i-1}t_rt_s) } \prod_{r\neq a,b,c} \prod_{s=b,c} \frac{1}{\Gamma( p^{\alpha_r+\alpha_s} t^{i-1}t_rt_s)}  \frac{1}{\theta( t^{i-1} t_bt_c x^{-d\beta};q)}.
\end{align*}
Now we can interchange limit and sum to obtain the desired result. Note that we are allowed to interchange sum and integral as the calculations in the appendix give us an absolutely summable bound on the summand. Indeed Lemma \ref{lempiv2} tells us that the summand is maximized for residues where the values of $z$ is originally around $p^{-(\alpha_1+\alpha_2+\alpha_3)}$ (in notation of that lemma), which is in this case $p^{\alpha_a-\beta}$ (this scale can be seen in the argument of the functions $f$ and $g$).
\end{proof}

The second proposition is applicable for $\alpha=(-\frac23,\frac13,\frac13,\frac13,\frac13,\frac13)$.
\begin{proposition}
Let $t_r\in \mathbb{C}$ be generic such that $t^{2(n-1)}t_0t_1t_2t_3t_4t_5=q$.
Let $\sum_{r=0}^5 \alpha_r=1$. Let $0\leq a\leq 5$ be such that $\alpha_a< -\frac12$ and $-\frac12<\alpha_r\leq \frac12$ for $r\neq a$. We write $\beta=\alpha_a+\frac12$.  We have 
\begin{align*}
\lim_{k\to \infty} &  p^{-\val(f)-\val(g)} \langle f(\cdot;t_rp^{\alpha_r}),g(\cdot;t_rp^{\alpha_r}) \rangle_{t_r p^{\alpha_r}}  = \prod_{i=1}^n \frac{1}{( qt^{2n-i-1} t_a^2 x^{2d\beta},q t^{i-1},  q t^{1-i} t_a^{-2} x^{-2d\beta};q)}
\\ & \times \sum_{\lambda}
 \lc(f)(q^{\lambda_j} x^{\beta d} t^{n-j} t_a )  \lc(g)(q^{\lambda_j} x^{\beta d} t^{n-j} t_a ) 
  q^{4n(\lambda')} t^{-2n(\lambda)}  \left(q t^{2(n-1)} t_a^4 x^{4d\beta} \right)^{|\lambda|}
\\ & \qquad \qquad \times 
\frac{\tilde D_{\lambda}(q,t) \tilde C_{2\lambda^2}^0(q t^{2(n-1)} t_a^2 x^{2d \beta}) }
{ \tilde C_{\lambda}^+(q t^{2n-3} t_a^2 x^{2d\beta}, t^{2(n-1)} t_a^2 x^{2d\beta})
\tilde C_{\lambda}^0 (q t^{n-2} t_a^2 x^{2d\beta}, t^{n-1} t_a^2 x^{2d\beta})}
\end{align*}
where $\lc( f)=\lc(f(zp^{-\frac12}))$, and likewise for $\lc(g)$ and the valuations. The summation is over all (weakly) ordered sequences $\lambda \in \mathbb{Z}^n$.
\end{proposition}
\begin{proof}
The proof is essentially the same as that of the previous proposition. However this time we use Lemma \ref{lempv} and Lemma \ref{lemlowa0} for the estimates. The maximum is at $\zeta=-\frac12$, so, since the series starts from $\alpha_a<-\frac12$ we have to move up $-\beta=(-\frac12) - \alpha_a$. Doing the calculations as before leads to the following expression for the sum of residues (where we once again assume $a\neq 4,5$ for notational convenience):
\begin{align*}
& \sum_{(\beta k d)^n \subset \lambda \subset (M+\beta k d)^n}
 f(q^{\lambda_j} x^{\beta d} t^{n-j} t_a p^{-\frac12})  g(q^{\lambda_j} x^{\beta d} t^{n-j} t_a p^{-\frac12}) 
q^{2n(\lambda')} t^{|\lambda|} 
\\ & \times 
\frac{D_{\lambda}(q,t) C_{2\lambda^2}^0(q t^{2(n-1)} t_a^2 x^{2d \beta}) C_{\lambda}^0(p^{-\beta} t^n x^{\beta d}) \prod_{r\neq a}  C_{\lambda}^0(p^{\frac12+\alpha_r} t^{n-1} t_at_r x^{\beta d})}
{ C_{\lambda}^+(q t^{2n-3} t_a^2 x^{2d\beta}, t^{2(n-1)} t_a^2 x^{2d\beta})
C_{\lambda}^0(p^{\beta+1} q t^{n-2} t_a^2 x^{d\beta}) \prod_{r\neq a} C_{\lambda}^0(p^{\frac12-\alpha_r} q t^{n-1} \frac{t_a}{t_r} x^{d\beta})}
\\ & \times 
\prod_{i=1}^n \frac{\Gamma(q t^{n-i},  p^{-\beta} t^{1+n-i} x^{d\beta}, p^{1+\beta} q^{\lambda_i+1} t^{n-i-1} t_a^2 x^{d\beta}, p^{\beta+1} q^{\lambda_i} t^{n-i} t_a^2 x^{d\beta}, q t^{n-i-1} t_a^2 x^{2 d \beta})  }
{\Gamma(p^{-\beta} q^{\lambda_i+1} t^{n-i} x^{d\beta}, p^{-\beta} q^{\lambda_i} t^{n+1-i} x^{d\beta}, p^{\beta+1} q t^{n-i-1} t_a^2 x^{d\beta}, q^{\lambda_i+1} t^{n-i-1} t_a^2 x^{2d\beta}) }
\\& \times \prod_{i=1}^n 
\frac{\Gamma(qt^{2n-i-1} t_a^2 x^{2d\beta}  )}{\Gamma(q^{\lambda_i} t^{n-i} t_a^2 x^{2d\beta} )}
\prod_{r\neq a} \frac{\Gamma( p^{\frac12+\alpha_r} t^{n-i} t_at_r x^{d\beta})}
{\Gamma( p^{1+\alpha_a+\alpha_r} t^{n-i} t_at_r, p^{\frac12-\alpha_r} q t^{n-i} x^{d\beta} \frac{t_a}{t_r})}
\prod_{0\leq r<s\leq 5:r,s\neq a} \frac{1}{\Gamma( p^{\alpha_r+\alpha_s} t^{i-1}t_rt_s)}
\\& \times 
\prod_{i=1}^n \frac{\theta( qt^{n-i-1} t_a^2 x^{2d\beta};q)}{\theta(q^{\lambda_i+1} t^{n-i-1} t_a^2 x^{2d\beta}, q^{\lambda_i} t^{n-i} t_a^2 x^{2d\beta};q)}
\end{align*}
Now we can interchange limit and sum to obtain the desired result. 
\end{proof}

\section{Integrals with more parameters}\label{sec6}
The cases we have considered so far are the evaluation cases: we have an explicit evaluation (as product of elliptic gamma functions) of the constant term. Adding extra parameters to the integrals still gives us interesting functions. This corresponds to looking at integrals  with $m>0$ in \cite{Rainstrafo}. For example the beta integral with $W(E_7)$ symmetry is of this form. It turns out that with the exact same argument as before the limits given in Section \ref{secac} are all still valid. 

In this section we will only present the results as the proofs are identical to the ones given before. Moreover, to simplify the notation we will not consider the associated bilinear forms, though extending the results to that case works in the same way as before. 

We thus consider the integral of Definition \ref{defhm} below. The prefactor of $p,q$-Pochhammer symbols has been chosen to ensure that the integral is holomorphic. Since we do not have an evaluation it is of course impossible to normalize the integral to the value 1 as we did in the previous section. If we were to turn the integral into a bilinear form the necessary prefactor to make the result holomorphic would depend on the functions $f$ and $g$ we plug in (it would essentially be the same prefactor but with $\tilde t_r$ instead of $t_r$), which is one of the reasons the equations become more convoluted in that case.
\begin{definition}\label{defhm}
Let $p,q,t\in \mathbb{C}^*$ satisfy $|t|,|p|,|q|<1$. For generic parameters $t_r\in \mathbb{C}^*$ satisfying the balancing condition $t^{2(n-1)} \prod_{r=0}^{2m+5} t_r=(pq)^{m+1}$ we define the integral
\begin{multline*}
\II_m^{(n)}(t_r;t,p,q) \\ :=
\prod_{i=1}^n\prod_{0\leq r<s\leq 2m+5} (t^{i-1}t_rt_s;p,q)
\frac{(q;q)^n (p;p)^n \Gamma(t)^n}{2^n n!} \int_{C^n} 
\prod_{1\leq j<k\leq n} \frac{\Gamma(tz_j^{\pm 1}z_k^{\pm 1})}{\Gamma(z_j^{\pm 1}z_k^{\pm 1})}
\prod_{j=1}^n \frac{\prod_{r=0}^{2m+5} \Gamma(t_r z_j^{\pm 1})}{\Gamma(z_j^{\pm 2})} 
\frac{dz_j}{2\pi i z_j}.
\end{multline*}
In the case that $|t_r|<1$ we define the contour $C$ to be the unit circle. By \cite[Theorem 10.7]{Rainstrafo} this function extends to a holomorphic function on 
$t_r \in \mathbb{C}^*$, $|t|,|p|,|q|<1$.
\end{definition}

The first propositions are given as follows. The analogue of Proposition \ref{propPI}:
\begin{proposition}\label{propPI2}
Choose generic parameters satisfying $t^{2(n-1)} \prod_r t_r=q^{m+1}$.
Let $\alpha \in \mathbb{R}^{2m+6}$, $\sum_{r=0}^{2m+5} \alpha_r=m+1$ and $0\leq \alpha_r\leq 1$ for $0\leq r\leq 2m+5$. 

We now have the limit
\begin{align*}
\lim_{p\to 0}   \II_m^{(n)}(t_r p^{\alpha_r}) &  = \frac{(q;q)^n  \prod_{j=1}^n \prod_{0\leq r<s\leq 2m+5: \alpha_r=\alpha_s=0} (t^{j-1} t_rt_s;q)}
{2^n n! \ (t;q)^n} \\ &\qquad  \times 
\int_{C^n} 
\prod_{1\leq j<k\leq n} \frac{(z_j^{\pm 1} z_k^{\pm 1};q)}{(tz_j^{\pm 1}z_k^{\pm 1};q)}
\prod_{j=1}^n \frac{(z_j^{\pm 2};q) \prod_{r:\alpha_r=1} (q t_r^{-1} z_j ^{\pm 1};q)}{\prod_{r:\alpha_r=0} (t_rz_j^{\pm 1};q)} \frac{dz_j}{2\pi i z_j},
\end{align*}
Here the integration contour $C=C^{-1}$ is such that it includes the points $q^j t_r$, (for $0\leq r\leq 2m+5$ with $\alpha_r=0$ and $j\geq 0$) excludes their reciprocals, and contains $q^jtC$ ($j\geq 0$). The contour can be taken to be the unit circle if $|t_r|<1$ for all $r$ with $\alpha_r=0$.
\end{proposition}
The analogue of Proposition \ref{proplimip2}
\begin{proposition}\label{proplimip2b}
Let $t_r\in \mathbb{C}$ be generic such that $t^{2(n-1)} \prod_r t_r=q^{m+1}$. 
Choose $\alpha\in \mathbb{R}^{2m+6}$, and $\zeta\in \mathbb{R}$.
Suppose $-\frac12\leq \zeta=\alpha_0=\alpha_1 <0$, $-\zeta\leq \alpha_r\leq 1+\zeta$ for $r>1$, and that $\sum_{r=0}^{2m+5} \alpha_r=m+1$

Then we have the limit
\begin{align*}
\lim_{p\to 0} \II_m^{(n)}(t_rp^{\alpha_r}) &   = 
\frac{(q;q)^n \prod_{j=1}^n (t^{n-j} t_0t_1;q)^{1_{\zeta=-1/2}} 
\prod_{r:\alpha_r=-\zeta} (t^{n-j} t_rt_0,t^{n-j}t_rt_1;q)}
{n! (t;q)^n  \theta(t^{n-j} t_0v,t^{n-j} t_1v;q)} \\ & \qquad \times 
\int_{C^n} \prod_{1\leq j<k\leq n} \frac{(z_j/z_k,z_k/z_j;q)}{(tz_j/z_k,tz_k/z_j;q)}
\left( \frac{(z_jz_k,qz_jz_k/t;q)  }{(tz_jz_k,qz_jz_k;q)  }   \right)^{1_{\zeta=-1/2}}
\\ & \qquad \qquad \qquad \times \prod_{j=1}^n \frac{\prod_{r:\alpha_r=1+\zeta} (qz_j/t_r;q)}{(t_0/z_j,t_1/z_j;q) \prod_{r: \alpha_r=-\zeta} (t_rz_j;q)}
\left( \frac{(z_j^2;q)}{(t_0z_j,t_1z_j, qz_j^2;q)}  \right)^{1_{\zeta=-1/2}} 
\\ & \qquad \qquad \qquad \qquad \qquad \times
\theta(vz_j, \frac{qz_j}{t^{n-1}vt_0t_1};q)
\frac{dz_j}{2\pi i z_j}
\end{align*}
for arbitrary $v\in \mathbb{C}^*$. Here the contour $C$ contains the points $q^j t_0$ and $q^j t_1$ (for $j\geq 0$), while excluding $q^{-j}/ t_r$ (for $j\geq 0$ and $r$ such that $\alpha_r=-\zeta$) and, if $\zeta=-1/2$, excluding $q^{-j}/ t_0$ and $q^{-j}/t_1$ for $j\geq 0$. Moreover $C$ should contain the contours $tq^jC$ (for $j\geq 0$).
\end{proposition}
The analogue of Proposition \ref{proplimip3}:
\begin{proposition}\label{proplimip3b}
Let $t_r\in \mathbb{C}$ be generic, satisfying $t^{2(n-1)} \prod_r t_r=q^{m+1}$. Suppose $\alpha \in \mathbb{R}^{2m+6}$ and $\zeta\in \mathbb{R}$ satisfy $-\frac12\leq \zeta<0$, $\sum_{r=0}^{2m+5} \alpha_r=m+1$, $\alpha_0+\alpha_1+\alpha_2 =\zeta$ and
$\zeta\leq \alpha_r\leq -\zeta$ for $r=0,1,2$ and
$-\zeta\leq \alpha_r\leq 1+\zeta$ for $r>2$. 
Then we have
\begin{align*}
\lim_{p\to 0} \II_m^{(n)}(t_rp^{\alpha_r}) & = 
\frac{(q;q)^n \prod_{j=1}^n  \prod_{\substack{r,s \leq 2\\ \alpha_r+\alpha_s=-1}} (t^{n-j} t_rt_s;q) 
\prod_{\substack{r\leq 2, s>2 \\ \alpha_r+\alpha_s=0}} (t^{n-j} t_rt_s;q)
 }{n! (t;q)^n }
\\& \qquad \times 
\int_{C^n} 
\prod_{1\leq j<k\leq n} \frac{(z_j/z_k,z_k/z_j;q)}{(tz_j/z_k,tz_k/z_j;q)}
\left( \frac{(z_jz_k,qz_jz_k/t;q)  }{(tz_jz_k,qz_jz_k;q)  }   \right)^{1_{\zeta=-1/2}}
\\& \qquad \qquad  \times \prod_{j=1}^n
\frac{\prod_{\substack{r\leq 2 \\ \alpha_r=- \zeta}} (q/t_rz_j;q) \prod_{\substack{r>2  \\ \alpha_r=1+\zeta}} (qz_j/t_r;q)}
{\prod_{\substack{r\leq 2 \\  \alpha_r=\zeta}} (t_r/z_j;q) \prod_{\substack{r>2  \\ \alpha_r=-\zeta}} (t_rz_j;q)}
\left( \frac{ (z_j^2;q) \prod_{\substack{r \leq 2 \\ \alpha_r=1/2}} (qz_j/t_r;q)}{(qz_j^2;q)\prod_{\substack{r \leq 2  \\ \alpha_r=-1/2}} (t_rz_j;q)}  \right)^{1_{\zeta=-1/2}}
\\ & \qquad \qquad \times \theta(qt^{1-n}z_j/t_0t_1t_2;q)
\frac{dz_j}{2\pi i z_j},
\end{align*}
where we have the usual conditions on the integration contour.
\end{proposition}
And finally the series limit, which is the analogue of Proposition \ref{propsumlim}.
\begin{proposition}\label{propsumlimb}
Let $t_r\in \mathbb{C}$ be generic such that $t^{2(n-1)}\prod_{r=0}^{2m+5}=q^{m+1}$.
Let $\sum_{r=0}^{2m+5} \alpha_r=m+1$. 
Suppose $-\frac12 \leq \alpha_0 <0$ and $1+\alpha_0\geq \alpha_r > \alpha_0$ for $r>0$ and such that $1\geq \alpha_r+\alpha_s> 0$ for $r,s\neq 0$. 
Moreover assume $2\alpha_0 = \sum_{r>0: \alpha_r+\alpha_0 <0} (\alpha_r+\alpha_0)$.

Then we get
\begin{align*}
\lim_{p\to 0} & \II_m^{(n)}(t_rp^{\alpha_r}) =
\prod_{j=0}^{n-1}  \frac{\prod_{r:\alpha_r=1+\alpha_0} (qt^{j} t_0/ t_r;q) } {(t^{n-j};q) (q t^{n-1+i} t_0^2;q)^{1_{\{\alpha_0=-1/2\}}} 
} 
\\ & \times 
\sum_{\lambda} 
\tilde \Delta_\lambda(t^{2(n-1)} t_0^2)^{1_{\{\alpha_0=-1/2\}}}
\left( \frac{\tilde C_{\lambda}^0(t^n)}{\tilde C_{\lambda}^-(q,t)} 
(-t^{5(n-1)}  t_0^4 q^2)^{-|\lambda|} q^{-3n(\lambda')} t^{5n(\lambda)} \right)^{1_{\{\alpha_0\neq -1/2\}}}
\\ &  \times 
 \frac{ \prod_{r> 0: \alpha_r=-\alpha_0} \tilde C_{\lambda}^0 (t^{n-1} t_0 t_r)}{
 \prod_{r> 0: \alpha_r=1+\alpha_0} \tilde C_{\lambda}^0(q t^{n-1} t_0/t_r)}
\left((-1)^{A+1} t^{(n-1)(A+3)} t_0^{A+2} q^{2} \prod_{r: \alpha_r+\alpha_0<0}  t_r  \right)^{|\lambda|} q^{(A+1) n(\lambda')} t^{-(A+1) n(\lambda)}
\end{align*}
where $A=|\{ r~|~ r>0,  \alpha_r<-\alpha_0\}|$.
\end{proposition}

Remark.  If $\alpha_1+\alpha_2=0$, say, then the above Proposition does not
apply, and one can either use Proposition \ref{proplimip3b} to get an
expression as an integral, or (assuming $|t_1t_2|<1$ for convergence) may
fix the series expression by including an additional factor $\prod_{0\le
  j<n} (t_1t_2t^{n-j-1};q)$.

\appendix
\section{Bounds on the integrand}\label{apb}
In this section we make explicit the bounds on the integrand used in the proof of Proposition \ref{propsumlim}. We say that a statement holds for $z$ away from the set $P$ (usually of zeros or poles of some function) if for all $\epsilon>0$ it holds for all $z$ such that $|1-z/p|>\epsilon$ for all $p\in P$.

\begin{lemma}\label{lemb1}
For all $M>0$ and all $q$ with $|q|<1$ there exist constants $C_1, C_2>0$ such that 
\[
C_1 \leq |(z;q)| \leq C_2
\]
for all $z$ with $|z|\leq M$ and $z$ away from the set of zeros of $(z;q)$. 
\end{lemma}
\begin{proof}
For $|z|\leq |q|^{1/2}$ we have the bound
\[
|(z;q)| = \prod_{k\geq 0} |1-zq^k| \leq \prod_{k\geq 0} (1+|z||q|^k) \leq \prod_{k\geq 0} (1+|q|^{k+1/2})
 = (-|q|^{1/2};|q|),
\]
and
\[
|(z;q)| =  \prod_{k\geq 0} |1-zq^k| \geq \prod_{k\geq 0} (1-|z||q|^k) \geq 
\prod_{k\geq 0} (1-|q|^{k+1/2})
= (|q|^{1/2};|q|).
\]
Thus we find for $M\leq |q|^{1/2}$ the bound 
\[
(|q|^{1/2};|q|) \leq |(z;q)| \leq (-|q|^{1/2};|q|).
\]

Suppose $|1-z/p| \geq \epsilon$ for all zeros of $(z;q)$. Then in particular we have $|1-z|\geq \epsilon$. Thus for $|q|^{1/2}\leq |z|\leq |q|^{-1/2}$ we get 
\[
|(z;q)| = |1-z| |(qz;q)| \leq |1-z| (-|q|^{1/2};|q|) \leq (1+|q|^{-1/2})(-|q|^{1/2};|q|) =
(-|q|^{-1/2};|q|) 
\]
and
\[
|(z;q)| = |1-z| |(qz;q)| \geq \epsilon (|q|^{1/2};|q|).
\]
Using induction we can subsequently easily prove that
for $|q|^{-n+1/2} \leq |z|\leq |q|^{-n-1/2}$ (with $n\geq 0$) we have (for $z$ away from the zeros of $(z;q)$) 
\[
\epsilon (-1)^{n} (|q|^{1/2-n};|q|)
\leq |(z;q)| \leq (-|q|^{-n-1/2};|q|).
\]
Note the $(-1)^n$ factor on the left hand side makes it positive.
\end{proof}

\begin{lemma}
For all $M>0$ and all $q$ with $|q|<1$ there exist constants $C_1, C_2>0$ such that 
for all $p$ with $|p|<|q|$ we have
\[
C_1 \leq |(z;p,q)| \leq C_2
\]
for all $z$ with $|z|\leq M$ and $z$ away from the set of zeros of $(z;p,q)$. 
\end{lemma}
\begin{proof}
The proof can be done in a very similar way to the one above, however we use a short cut using the above result. First we note that for $|z|\leq |q|^{1/2}$ we have
\[
|(|q|^{1/2};|q|,|q|)| \leq |(z;p,q)|  \leq |(-|q|^{1/2};|q|,|q|)|,
\]
which we prove as in the previous lemma.
Then we use that for $|z|\leq |q|^{1/2-n}$ (for $n\geq 0$) we get 
\[
|(z;p,q)| = |(zp^{n};p,q)| \prod_{r=0}^{n-1} |(zp^r;q)|.
\]
By our assumption on $|p|$ we see $|zp^{n}| \leq |z||q|^{n} \leq |q|^{1/2}$, so the first part can be bounded (above and below), and the remaining product is a finite product (the length of which is independent of $z$ as long as $|z|\leq |q|^{-n}$) and can be bounded using the above proposition. 
\end{proof}

\begin{lemma}\label{lemb3}
For all $M>0$ and all $\alpha>0$ and all $q$ with $|q|<1$ there exist constants $C_1, C_2>0$  such that for all 
$p$ with $|p|<|q|$ we have
\[
C_1 \leq |(zp^{-\alpha};q)|  |z|^{-k\alpha} |q|^{\binom{k\alpha+1}{2}} \leq  C_2,
\]
for all $z$ away from the zeros of $(zp^{-\alpha};q)$ and of $(qp^{\alpha}/z;q)$ with $1/M<|z|<M$, where we write $p=xq^k$ for some
$x\in \mathbb{C}$ with $|x|=1$ and $k\in \mathbb{R}$.
\end{lemma}
Note that for $\alpha\leq 0$ we can use the first lemma to see that there exist constants $C_1$ and $C_2$ such that
\[
C_1 \leq |(zp^{-\alpha};q)|   \leq  C_2.
\]
\begin{proof}
We write
\begin{align*}
(zp^{-\alpha};q) 
&= (zx^{-\alpha} q^{-k\alpha};q) 
= (zx^{-\alpha} q^{-\{k\alpha\}};q)(zx^{-\alpha} q^{-k\alpha};q)_{\lfloor k\alpha\rfloor} 
\\ &= (zx^{-\alpha} q^{-\{k\alpha\}};q) (q^{1+\{k\alpha\}} x^{\alpha}/z;q)_{\lfloor k\alpha\rfloor}
(-zx^{-\alpha} q^{-k\alpha})^{\lfloor k\alpha\rfloor} q^{\binom{\lfloor k\alpha\rfloor}{2}} 
\\ &= \frac{(zx^{-\alpha} q^{-\{k\alpha\}};q) (q^{1+\{k\alpha\}} x^{\alpha}/z;q)}{
 (q^{1+k\alpha} x^{\alpha}/z;q)} (-zx^{-\alpha} q^{-k\alpha})^{\lfloor k\alpha\rfloor} q^{\binom{\lfloor k\alpha\rfloor}{2}}.
\end{align*}
The three $q$-Pochhammer symbols in the final expression all have arguments which are bounded by
$M/q$, resp. $Mq$, resp. $Mq$, so we can bound those (away from the zeros) by the previous lemma.
For the remaining part we note that $|-x^{-\alpha}|=1$, so that does not change the norm and that
\[
\binom{\lfloor k\alpha \rfloor}{2} - k\alpha \lfloor k\alpha \rfloor = 
- \binom{k\alpha+1}{2} + \frac12 \{k\alpha\}^2.
\]
Combining everything together gives the desired bound.
\end{proof}

\begin{lemma}\label{lemb4}
Consider the integrand
\[
I(z) = 
\frac{(q;q)^n (p;p)^n \Gamma(t;p,q)^n}{2^n n! \prod_{j=1}^n\Gamma(t^j;p,q) \prod_{ r<s} \Gamma(t^{n-j} t_rt_s;p,q)} 
 \prod_{1\leq j<k\leq n} \frac{\Gamma(tz_j^{\pm 1} z_k^{\pm 1};p,q)}{\Gamma(z_j^{\pm 1}z_k^{\pm 1};p,q)}
\prod_{j=1}^n \frac{\prod_{r=0}^5 \Gamma(t_rz_j^{\pm 1};p,q)}{\Gamma(z_j^{\pm 2};p,q)} \frac{1}{2\pi i z_j}.
\]
Let $\alpha$ be in the polytope $\hat P_{II}$ (which is $P_{II}$ from \cite{vdBR} except that we insist that some bounding inequalities are strict) given by the equations 
\begin{multline*}
-\frac12\leq \alpha_0<0, \qquad  1+\alpha_0\geq \alpha_r > \alpha_0, \quad (1\leq r\leq 5), \qquad 1\geq \alpha_r+\alpha_s\geq 0, \quad (1\leq r<s\leq 5), \\
2\alpha_0 = \sum_{r>0: \alpha_r+\alpha_0 <0} (\alpha_r+\alpha_0), \qquad \sum_r \alpha_r=1.
\end{multline*}
Write $z_i = x_i p^{\zeta_i}$ with $|x_i|=1$. Then there exist constants $C_1,C_2>0$ such that 
for all $p$ and all $\alpha_0\leq \zeta_i \leq -\alpha_0$ we have away from the zeros and poles of $I(z)$ 
\[
C_1 \leq |I(z)| |q|^{l^2 c(\zeta)} \left( d(t_r;q,t)  \right)^l  \leq C_2
\]
where we write $p=yq^l$ for some $|y|=1$ (thus $l=\log_{|q|}(|p|)$). Here $c(\zeta)$ is given by
\[
c(\zeta) = \frac12 \sum_i \left[
2\zeta_i^2 - 2 \alpha_0^2 + 
\sum_{r\geq 1: \alpha_r<-\alpha_0} (\alpha_r+\alpha_0)^2 - 
\sum_{r\geq 1: \alpha_r< -|\zeta_i|} (2\alpha_r^2 + 2\zeta_i^2) -\sum_{r\geq 1:-|\zeta_i|\leq \alpha_r< |\zeta_i|} (\alpha_r-|\zeta_i|)^2 \right].
\]
and $d(t_r;q,t)$ by 
\begin{multline*}
d(t_r;q,t) = \prod_{i=1}^n |t|^{2\zeta_i(i-1) + 2\alpha_0(n-i)}
|q|^{\sum_{r\geq 1:\alpha_r<-\zeta_i} \alpha_r + \frac12 \sum_{r\geq 1:-\zeta_i\leq \alpha_r<\zeta_i} (\alpha_r-\zeta_i)}
\\ \times \prod_{r\geq 1: \alpha_r<-\zeta_i} |t_r|^{\alpha_0-\alpha_r}
\prod_{r\geq1 : -\zeta_i\leq \alpha_r<\zeta_i} |t_r|^{\alpha_0+\zeta_i} \prod_{r\geq 1: \zeta_i\leq \alpha_r< -\alpha_0} |t_r|^{\alpha_0+\alpha_r} 
\end{multline*}
when the $\zeta_i$ are ordered such that $\zeta_1\leq \zeta_2\leq \cdots \leq \zeta_n$.
\end{lemma}
\begin{proof}
By symmetry we may assume that $0\leq \zeta_1\leq \zeta_2\leq \cdots \leq \zeta_n$. 
Removing the constants from $I(z)$, replacing $\Gamma(z)$ by $(pq/z;p,q)/(z;p,q)$, replacing 
$(p^{\alpha}x;p,q)$ for $\alpha<0$ and $x$ independent of $p$ by $(p^{\alpha}x;q) (p^{\alpha+1};p,q)$ and subsequently removing all $(p^{\alpha}x;p,q)$-terms with $\alpha\geq 0$ (as by the second lemma those are bounded above and below) we find that there exist constants such that outside of poles we have
\[
C_1 \leq \left| I(z) 
\prod_{1\leq i<j\leq n} \frac{(tz_i/z_j,t/z_iz_j;q)}{(z_i/z_j,1/z_iz_j;q)}
\prod_{i=1}^n \frac{\prod_{r\geq 0: \alpha_r<-\zeta_i} (t_r z_i;q) \prod_{r\geq 0: \alpha_r<\zeta_i} (t_r/z_i;q)  }{(z_i^{-2};q) \prod_{r\geq 1: \alpha_r+\alpha_0<0} (t^{n-i}t_rt_0;q) }  z_i \right|
   \leq C_2.
\]
Now, replacing $z_j$ by $x_j p^{\zeta_j}$ and $t_r$ by $u_r p^{\alpha_r}$ and using the final lemma
we get that there exist (different from before) constants such that
\[
C_1 \leq \left| I(z) 
\prod_{1\leq i<j\leq n}
|t|^{2l\zeta_j} 
\prod_{i=1}^n \frac{\prod_{r\geq 0: \alpha_r<-\zeta_i}
|t_r|^{-l(\alpha_r+\zeta_i)} q^{-\binom{1-l(\alpha_r+\zeta_i)}{2}} \prod_{r\geq 0: \alpha_r<\zeta_i} 
|t_r|^{-l(\alpha_r-\zeta_i)} q^{-\binom{1-l(\alpha_r-\zeta_i)}{2}} }{
q^{-\binom{1+2l\zeta_i}{2}}
 \prod_{r\geq 1: \alpha_r+\alpha_0<0} \left(|t|^{n-i} |t_r||t_0| \right)^{-l(\alpha_r+\alpha_0)}
 q^{-\binom{1-l(\alpha_r+\alpha_0)}{2}}  }  p^{\zeta_i} \right|
   \leq C_2
\]
which simplifies to the desired result.

\end{proof}
\begin{lemma}
For $\alpha$ in $\hat P_{II}$ and $\alpha_0 \leq \zeta_i\leq -\alpha_0$ we have $c(\zeta)\leq 0$ (where $c(\zeta)$ is given in the previous lemma)
and equality holds only if either $\alpha_r+\alpha_s=0$ for some $r,s\geq 1$, or $\zeta_i=\pm \alpha_0$ (for all $i$). 

\end{lemma}
\begin{proof}
It is easy to check equality holds if $\alpha_r+\alpha_s=0$ for some $r,s\geq 1$ (which implies that
$\alpha_j\geq -\alpha_0$ for all $j\neq 0,r,s$), or if $\zeta=\pm \alpha_0$. 

Notice that $c(\zeta)$ is even in the $\zeta_i$, so we can assume $\zeta_i \geq 0$. Let us consider the function $h$ given by 
\[
h(\zeta) =  2 \zeta^2 - 2 \alpha_0^2 + \sum_{r\geq 1: \alpha_r<-\alpha_0} (\alpha_r+\alpha_0)^2 - 
\sum_{r\geq 1: \alpha_r< -\zeta} (2\alpha_r^2 + 2\zeta^2) -\sum_{r\geq 1:-\zeta\leq \alpha_r< \zeta} (\alpha_r-\zeta)^2,
\]
so $c(\zeta) = \frac12 \sum_i h(\zeta_i)$. 
Now there are two options
\begin{itemize}
\item There exists an $r$ such that $\alpha_r<-\zeta$.

In this case we notice that for all $s\neq r$ we have $\alpha_s \geq -\alpha_r > \zeta$. Thus in this case the function becomes
\begin{align*}
h(\zeta) & = 2\zeta^2 - 2\alpha_0^2 + \sum_{t\geq 1: \alpha_t<-\alpha_0} (\alpha_t+\alpha_0)^2 
- (2\alpha_r^2+2\zeta^2)  \\
&=  -2\alpha_r^2 - 2\alpha_0^2 + (\alpha_r+\alpha_0)^2 + \sum_{t\geq 1, t\neq r: \alpha_t<-\alpha_0} (\alpha_t+\alpha_0)^2  \\ 
&=  - (\alpha_0-\alpha_r)^2 + \sum_{t\geq 1, t\neq r: \alpha_t<-\alpha_0} (\alpha_t+\alpha_0)^2.
\end{align*}
Now we observe that $\alpha_0-\alpha_r = 2\alpha_0 - (\alpha_r+\alpha_0) = 
\sum_{t\geq 1, t\neq r, \alpha_t<-\alpha_0} (\alpha_t+\alpha_0) $. Plugging this in and noticing that the terms $\alpha_t+\alpha_0$ are all negative (thus in particular all have the same sign), we immediately see that $h(\zeta)\leq 0$.

\item There exist no $r$ such that $\alpha_r<-\zeta$.

In this case we write (using $2=\sum_{r\geq 1: \alpha_r<-\alpha_0} \frac{\alpha_r+\alpha_0}{\alpha_0}$)
\begin{align}
h(\zeta) &= 2\zeta^2 - 2 \alpha_0^2 + \sum_{r\geq 1: \alpha_r<-\alpha_0} (\alpha_r+\alpha_0)^2  -\sum_{r\geq 1:-\zeta\leq \alpha_r< \zeta} (\alpha_r-\zeta)^2 \nonumber  \\ 
&= \sum_{r\geq 1: \alpha_r<- \alpha_0} \left[
\frac{\zeta^2 (\alpha_r+\alpha_0)}{\alpha_0} - (\alpha_r+\alpha_0) \alpha_0 +  (\alpha_r+\alpha_0)^2  -
1_{\{\alpha_r<\zeta\}}  (\alpha_r-\zeta)^2 \right].  \label{eqgzeta}
\end{align}
Now let $m=\min(\alpha_r ,\zeta)$, then we can rewrite the summands as
\begin{align*}
\frac{\zeta^2 (\alpha_r+\alpha_0)}{\alpha_0}  &- (\alpha_r+\alpha_0) \alpha_0 +  (\alpha_r+\alpha_0)^2  -
1_{\{\alpha_r<\zeta\}}  (\alpha_r-\zeta)^2
\\ &= \frac{\zeta^2 (\alpha_r+\alpha_0)}{\alpha_0}  - (\alpha_r+\alpha_0) \alpha_0 +  (\alpha_r+\alpha_0)^2  -  (\zeta-m)^2
\\ &= \frac{(\alpha_r \zeta+\alpha_0 m)(\alpha_0+\zeta) }{\alpha_0}
+  (\alpha_r-m) \left(m - \zeta +\alpha_0 + \alpha_r \right)
\\ & \leq \frac{(\alpha_r \zeta+\alpha_0 m)(\alpha_0+\zeta) }{\alpha_0}.
\end{align*}
Now we have a few different cases
\begin{itemize}
\item $\alpha_r\geq 0$ for all $r\geq 1$. In this case all terms in the sum \eqref{eqgzeta} are negative, so the sum is as well.
\item There exist $r\neq s\geq 1$ with $\alpha_r,\alpha_s\leq \zeta$. In this case we have 
(since $\alpha_r+\alpha_s \geq 0$)
\[
(\alpha_r \zeta+\alpha_0 \min(\alpha_r,\zeta)) + (\alpha_s \zeta+\alpha_0 \min(\alpha_s,\zeta))
= (\alpha_r+\alpha_s)(\zeta+\alpha_0) \leq 0,
\]
while all the remaining summands are negative as before (we have at most one negative $\alpha_r$ with $r\geq 1$). 
\item There exists $\alpha_t<0$ for some $t\geq 1$ and $\alpha_s>\zeta$ for all $s\neq t$. In this case we get 
\[
\sum_{r:\alpha_r<-\alpha_0} \alpha_r \zeta+\alpha_0 m 
=  \alpha_t \zeta + \alpha_0 \alpha_t + \sum_{r \neq t: \alpha_r<-\alpha_0}
(\alpha_r + \alpha_0) \zeta
=  \alpha_t \zeta + \alpha_0 \alpha_t + \zeta (\alpha_0 - \alpha_t)
= \alpha_0 (\zeta+\alpha_t) \leq 0.
\]
So in this case as well we find that $h(\zeta)\leq 0$.
\end{itemize}
\end{itemize}
The cases when equality holds can also be easily determined from this analysis. 
\end{proof}


\begin{proposition}
The sum of the remaining integrals after picking up the residues in the proof of Proposition \ref{propsumlim} vanishes as $p\to 0$. 
\end{proposition}
\begin{proof}
First of all we note that the contours of the integrals can be chosen inside the annulus around 0 with radii $p^{\alpha_a+\epsilon}$ and
$p^{-\alpha_a-\epsilon}$ (for the $\epsilon>0$ of the proof of Proposition \ref{propsumlim}) and away from the poles of the integrands. Moreover for generic parameters the length of the contour is at worst $\mathcal{O}(1/|p|)$ (that is, we might have to curve from near the circle with radius $p^{\alpha_a+\epsilon}$ to near the circle with radius $p^{-\alpha_a-\epsilon}$). The residues of the integrand $I(z)$ at $z_j = q^{k_j} t \tilde t_a$ (for $1\leq j\leq s$, for some value of $s\leq n$), satisfy a similar bound as the one on $I(z)$ itself given in Lemma \ref{lemb4}, since those residues have the same form as $I(z)$ with different parameters. Indeed the only difference in the bound is the explicit form of the function $d$, and the fact that in the definition of $c$ we only sum over those $i$ corresponding to $z_i$'s we have not yet taken a residue in.

Multiplying this integrand by functions $f$ and $g$ still shows that we have a bound of the order $|p|^{\rho \log_{|q|}(|p|)}$ for some $\rho>0$ (as the function $c$ from Lemma \ref{lemb4} is strictly positive in the domain we are interested in). Even if we multiply this bound by the length of the contour ($\approx 1/|p|^n$) and the number of integrals ($\approx \log(|p|)^n$) it still converges to 0 as $p\to 0$. 
\end{proof}

\subsection{Polytopes for bilateral series}
For the two propositions from Section \ref{sec5} the main ideas of the bounds are the same, but the practical calculations are slightly different. 

Define the polytope $P_{IV}$ by the equations
\[
\alpha_4+\alpha_5\geq 1, \qquad
\alpha_r+\alpha_s\geq 0, \quad (1\leq r<s\leq 3), \qquad
\alpha_0+\alpha_r\geq 0, \quad (r=4,5), \qquad 
\sum_{r=0}^5 \alpha_r=1
\]
\begin{lemma}
The vertices of $P_{IV}$ are the $S_1\times S_3\times S_2$ orbits of $(0,0,0,0,0,1)$, $(-\frac12,-\frac12,\frac12,\frac12,\frac12,\frac12)$ and 
$(-1,0,0,0,1,1)$. 
\end{lemma}
%

\begin{lemma}\label{lempiv}
Let $\alpha$ be in the polytope $P_{IV}$. Write $z_i = x_i p^{\zeta_i}$ with $|x_i|=1$. Then there exist constants $C_1,C_2>0$ such that 
for all $p$ and all $\alpha_0\leq \zeta_i \leq -\alpha_0$ we have away from the zeros and poles of $I(z)$ 
\[
C_1 \leq |I(z)| |q|^{l^2 c_{IV}(\zeta)} \left( d_{IV}(t_r;q,t)  \right)^l  \leq C_2
\]
where we write $p=yq^l$ for some $|y|=1$ (thus $l=\log_{|q|}(|p|)$). Here $c_{IV}(\zeta)$ is given by
\begin{multline*}
c_{IV}(\zeta) = \frac12 \sum_i \bigg[
2\zeta_i^2 - 2 \alpha_0^2 - (1-\alpha_4-\alpha_5)^2 - (|\zeta_i| - \alpha_0-1)^2 1_{\{1+\alpha_0<|\zeta_i|\}}
+\sum_{r\geq 1: \alpha_r<-\alpha_0} (\alpha_r+\alpha_0)^2  \\ 
+ \sum_{r\geq 1: \alpha_r>1-|\zeta_i|} (|\zeta_i|+\alpha_r-1)^2
- 
\sum_{r\geq 1: \alpha_r< -|\zeta_i|} (2\alpha_r^2 + 2\zeta_i^2) -\sum_{r\geq 1:-|\zeta_i|\leq \alpha_r< |\zeta_i|} (\alpha_r-|\zeta_i|)^2 \bigg].
\end{multline*}
and $d_{IV}(t_r;q,t)$ is some explicit product of powers of $|q|$, $|t|$ and $|t_r|$'s. 
\end{lemma}
\begin{proof}
The proof is very similar to that of Lemma \ref{lemb4}. Except now we use that, assuming $0\leq \zeta_1\leq \cdots \leq \zeta_n \leq -\alpha_0$ there exist constants bounding $I(z)$ as 
\begin{multline*}
C_1 \leq \bigg| I(z) 
\prod_{1\leq i<j\leq n} \frac{(\frac{tz_i}{z_j},\frac{t}{z_iz_j}, \frac{pt}{z_iz_j}, \frac{pq}{z_iz_j} ;q)}{(\frac{z_i}{z_j},\frac{1}{z_iz_j},\frac{p}{z_iz_j}, \frac{pq}{tz_iz_j};q)}
\\ \times 
\prod_{i=1}^n \frac{(pqz_i^{-2},pq t^{j-n}/t_4t_5,pt_0/z_i;q) \prod_{r=0}^3 (t_rz_i, t_r/z_i;q)}
{(z_i^{-2},pz_i^{-2};q) \prod_{r=4}^5 (\frac{pq}{t_rz_i};q) \prod_{r= 1}^3 (t^{n-i}t_rt_0;q) }  z_i \bigg|
   \leq C_2.
\end{multline*}
\end{proof}

\begin{lemma}\label{lempiv2}
For $\alpha_0 \leq \zeta_j \leq -\alpha_0$ we have $c_{IV}(\zeta) \leq 0$ with equality only if $\zeta_i=\pm (\alpha_1+\alpha_2+\alpha_3)$ for all $i$, or $\alpha_4=\alpha_5=-\alpha_0$ and $|\zeta_i| \geq \alpha_1+\alpha_2+\alpha_3$  for all $i$.
\end{lemma}
\begin{proof}
By evenness we may again assume $\zeta_i\geq 0$ for all $i$. Considering the different inequalities valid in $P_{IV}$ we obtain that 
\begin{multline*}
c_{IV}(\zeta) = \sum_{i=1}^n \bigg[ -(\zeta_i - (\alpha_1+\alpha_2+\alpha_3))^2 - (\zeta_i-\alpha_0-1)^2 1_{\{1+\alpha_0<\zeta_i\}} \\- 
\sum_{1\leq r\leq 3: \alpha_r+\zeta_i<0} (\alpha_r+\zeta_i)^2 
+\sum_{r\geq 4: \alpha_r> 1-\zeta_i} (\zeta_i+\alpha_r-1)^2
\bigg].
\end{multline*}
Note that if $\zeta_i=\alpha_1+\alpha_2+\alpha_3$, the equations for the polytope imply that the only term is 
$-(\zeta_i-(\alpha_1+\alpha_2+\alpha_3))^2$, which of course vanishes at this value of $\zeta_i$. The result now follows as $\zeta_i+\alpha_4-1\leq \zeta_i-\alpha_0-1$ and 
$\zeta_i+\alpha_5-1 \leq \zeta_i-(\alpha_1+\alpha_2+\alpha_3)$, so the only two positive terms are always less in absolute value than two given negative terms. 

This analysis shows that if $\zeta_i \neq \alpha_1+\alpha_2+\alpha_3$, the expression can only vanish if 
$\zeta_i+\alpha_4-1= \zeta_i+\alpha_5-1=\zeta_i-\alpha_0-1 = \zeta_i-(\alpha_1+\alpha_2+\alpha_3) >0$.
\end{proof}

The second polytope associated to a bilateral series is $P_V$ given by the bounding inequalities
\[
\alpha_0\leq -\frac12, \qquad \alpha_r \leq \frac12, \quad (r\geq 1), \qquad
\sum_{r=0}^5 \alpha_r=1.
\]
In this case the vertices are the $S_1\times S_5$ orbits of $(-\frac12,-\frac12,\frac12,\frac12,\frac12,\frac12)$ and 
$(-\frac32,\frac12,\frac12,\frac12,\frac12,\frac12)$. For bounding the integrand we find that 
\begin{lemma}\label{lempv}
Let $\alpha$ be in the polytope $P_{V}$. Write $z_i = x_i p^{\zeta_i}$ with $|x_i|=1$. Then there exist constants $C_1,C_2>0$ such that 
for all $p$ and all $\alpha_0\leq \zeta_i \leq -\alpha_0$ we have away from the zeros and poles of $I(z)$ 
\[
C_1 \leq |I(z)| |q|^{l^2 c_{V}(\zeta)} \left( d_{V}(t_r;q,t)  \right)^l  \leq C_2
\]
where we write $p=yq^l$ for some $|y|=1$ (thus $l=\log_{|q|}(|p|)$). Here $c(\zeta)$ is given by
\begin{multline*}
c_{V}(\zeta) = \frac12 \sum_i \bigg[
4\zeta_i^2 - (\alpha_0-|\zeta_i|)^2 -(\alpha_0+|\zeta_i|)^2 1_{\{\alpha_0+|\zeta_i|<0\}}
  - (1+\alpha_0-|\zeta_i|)^2  1_{\{1+\alpha_0<|\zeta_i|\}} 
\\ -(1+\alpha_0+|\zeta_i|)^2 1_{\{1+\alpha_0+|\zeta_i|<0\}}
-(2+\alpha_0-|\zeta_i|)^2 1_{\{2+\alpha_0<|\zeta_i|\}} 
+
\sum_{r=1}^5  \bigg( (\alpha_r+\alpha_0)^2 \\ -(\alpha_r+|\zeta_i|)^2 1_{\{\alpha_r+|\zeta_i|<0\}}
- (\alpha_r-|\zeta_i|)^2 1_{\{\alpha_r<|\zeta_i|\}}
 - (1+\alpha_r-|\zeta_i|)^2 1_{\{1+\alpha_r<|\zeta_i|\}}
+(1-\alpha_r-|\zeta_i|)^2 1_{\{1-\alpha_r<|\zeta_i|\}} \bigg)
 \bigg].
\end{multline*}
and $d_{V}(t_r;q,t)$ is a product of powers of $|t|$, $|q|$, $|t_r|$, the coefficients of which can be explicitly expressed in terms of $\alpha_r$ and $\zeta_i$.
\end{lemma}
\begin{proof}
The proof is again nearly identical to that of Lemma \ref{lemb4}. Except now we use that there exist constants bounding $I(z)$ as 
\begin{multline*}
C_1 \leq \bigg| I(z) 
\prod_{1\leq i<j\leq n} \frac{(tz_i/z_j,ptz_i/z_j,pqz_j/z_i,t/z_iz_j,pt/z_iz_j,p^2t/z_iz_j,pq/z_iz_j,p^2q/z_iz_j;q)}{(z_i/z_j,pz_i/z_j,pqz_j/tz_i,1/z_iz_j,p/z_iz_j,p^2/z_iz_j,pq/tz_iz_j,p^2q/tz_iz_j;q)}
\\ \times \prod_{i=1}^n 
\frac{(pq z_i^{-2}, t_0z_i^{\pm 1},pt_0z_i^{\pm 1}, p^2t_0z_i^{-1};q)   }
{(z_i^{-2},p z_i^{-2};q)  }  
\prod_{r=1}^5 \frac{(t_r z_i^{\pm 1}, pt_rz_i^{-1};q) }{ (t^{n-i}t_rt_0,pq/t_rz_i;q)}
z_i \bigg|
   \leq C_2.
\end{multline*}
\end{proof}

\begin{lemma}\label{lemlowa0}
For $\alpha_0 \leq \zeta_j \leq -\alpha_0$ we have $c_{V}(\zeta) \leq 0$ with equality only if $|\zeta_i|= \frac12$ or $|\zeta_i|=\frac32$ for all $i$, or if at least 4 of the bounding inequalities are satisfied (i.e. at least four of  $\alpha_0=-\frac12$, and $\alpha_r=\frac12$ ($r\geq 1$) hold).
\end{lemma}
\begin{proof}
By evenness we may again assume $\zeta_i\geq 0$ for all $i$. Considering the different inequalities valid in $P_{V}$ we obtain that for $0\leq \zeta_i\leq \frac12$ we get 
\begin{multline*}
c_{V}(\zeta) = \sum_{i=1}^n \bigg[ -(1-2\zeta_i)^2 + 1_{\{1+\alpha_0>\zeta_i\}} (1+\alpha_0-\zeta_i)^2
-(1+\alpha_0+\zeta)^2 1_{1+\alpha_0+\zeta_i<0}
\\ + \sum_{r\geq 1} (\alpha_r-\zeta_i)^2 1_{\{\alpha_r>\zeta_i\}}  - (\alpha_r+\zeta_i)^21_{\{\alpha_r+\zeta_i<0\}}  \bigg].
\end{multline*}
All the indicator functions vanish at $\zeta_i=\frac12$, so for those values of $\zeta_i$ this term is clearly zero.
The term between brackets is a piecewise quadratic function, the first derivative of which vanishes in 0 and $\frac12$, the second derivative of which is always an even integer and is negative in $\zeta_i=\frac12$, ($\left(\frac{d}{d\zeta_i}\right)^2 -(1-2\zeta_i)^2 = -8$), and changes sign at most once. 
Indeed, the second derivative only increases (if $\zeta_i$ increases) at points where $\zeta_i=-\alpha_r$ (for some $r\geq 1$) or $\zeta_i=-1-\alpha_0$, which is at most once. At such a point the second derivative increases by 2, and as it is always an even integer, it cannot go from strictly negative to strictly positive. In particular the only sign change the second derivative can make is changing from positive to negative once. These arguments imply that the derivative is always positive, and thus that the value in the interval $\zeta_i\in [0,\frac12]$ is maximized at $\zeta_i=\frac12$. We moreover have that this term is constant in some neighborhood of $\zeta_i=\frac12$, whenever at least 4 of the equations $\alpha_0=-\frac12$ and $\alpha_r=\frac12$ (for $r\geq 1$) hold. In the case 4 of these equations hold, the expression between brackets reduces to zero identically. 

For the regions $\zeta_i \in [\frac12,1]$, respectively $\zeta_i\in [1,\frac32]$ it is opportune to prove
that the value of $c_V$ does not change if we replace $\zeta_i$ by $1-\zeta_i$, respectively $1+\zeta_i$ (assuming the original $\zeta_i\in [0,\frac12]$). 
\end{proof}

\end{document}